\documentclass[11pt, a4paper]{article}
\usepackage{graphicx}
\usepackage{hyperref}
\usepackage{amsfonts, amsmath, amsthm}
\usepackage{algorithm}
\usepackage[noend]{algpseudocode}
\usepackage{color}
\usepackage{tikz}
\usepackage[shortlabels]{enumitem}
\usepackage{doi}
\usepackage{breakurl} 

\newcommand{\email}[1]{\href{mailto:#1}{#1}}

\newtheorem{theorem}{Theorem}
\newtheorem{observation}[]{Observation}
\newtheorem{lemma}[]{Lemma}

\newtheorem{corollary}[]{Corollary}
\newtheorem{proposition}[]{Proposition}
\newtheorem{conjecture}[]{Conjecture}
\newtheorem{definition}[]{Definition}

\begin{document}
\title{The Frank number and nowhere-zero flows on graphs\footnote{An extended abstract of this work appeared in the proceedings of the 49th International Workshop on Graph-Theoretic Concepts in Computer Science (WG 2023)~\cite{WG2023}.}}

\author{{\sc Jan GOEDGEBEUR\footnote{Department of Computer Science, KU Leuven Campus Kulak-Kortrijk, 8500 Kortrijk, Belgium}\;\footnote{Department of Applied Mathematics, Computer Science and Statistics, Ghent University, 9000 Ghent, Belgium}\;,
Edita MÁ\v{C}AJOVÁ\footnote{Comenius University, Mlynská dolina, 842 48 Bratislava, Slovakia}\;, }\\[1mm]
and {\sc Jarne RENDERS\footnotemark[1]}\;\footnote{E-mail: \email{jan.goedgebeur@kuleuven.be}; \email{macajova@dcs.fmph.uniba.sk}; \email{jarne.renders@kuleuven.be}}}

\date{}

\maketitle   
\begin{center}
\begin{minipage}{125mm}
{\bf Abstract.} An edge $e$ of a graph $G$ is called \emph{deletable} for some orientation~$o$ if the restriction of $o$ to $G-e$ is a strong orientation. Inspired by a problem of Frank, in 2021 H\"orsch and Szigeti proposed a new parameter for $3$-edge-connected graphs, called the Frank number, which refines $k$-edge-connectivity. The \textit{Frank number} is defined as the minimum number of orientations of $G$ for which every edge of $G$ is deletable in at least one of them. They showed that every $3$-edge-connected graph has Frank number at most $7$ and that in case these graphs are also $5$-edge-colourable the parameter is at most $3$. 
Here we strengthen both results by showing that every $3$-edge-connected graph has Frank number at most $4$ and that every graph which is $3$-edge-connected and $3$-edge-colourable has Frank number~$2$. The latter also confirms a conjecture by Bar\'at and Bl\'azsik. 
Furthermore, we prove two sufficient conditions for cubic graphs to have Frank number~$2$ and use them in an algorithm to computationally show that the Petersen graph is the only cyclically $4$-edge-connected cubic graph
up to $36$ vertices having Frank number greater than $2$. 

\bigskip


\medskip


\end{minipage}
\end{center}

\section{Introduction}
\label{sect:intro}

An \emph{orientation} $(G,o)$ of a graph $G$ is a directed graph with vertices $V(G)$ such that each edge $uv\in E(G)$ is oriented either from $u$ to $v$ or from $v$ to $u$ by the function $o$.
An orientation is called \emph{strong} if, for every pair of distinct vertices $u$ and $v$, there exists an oriented $uv$-path, i.e.\ an oriented path starting at vertex $u$ and ending at vertex $v$. It is not difficult to see that an orientation is strong if and only if each edge cut contains edges oriented in both directions.

An edge $e$ is \emph{deletable} in an orientation $(G,o)$ if the restriction of $o$ to $E(G) - \{e\}$ yields a strong orientation of $G - e$. 
Note that this implies that $o$ is a strong orientation.
In this paper, a \emph{circuit} is a connected $2$-regular graph, while a \emph{cycle} is a graph in which every vertex has even degree.
A graph in which the removal of fewer than $k$ edges cannot separate the graph into two components which both contain a cycle, is called \emph{cyclically $k$-edge-connected}. The \emph{cyclic edge connectivity} of a graph $G$ is the largest $k$ for which $G$ is cyclically $k$-edge-connected. 

Inspired by a problem of Frank, in 2021, H\"{o}rsch and Szigeti~\cite{HS21} proposed a new parameter for $3$-edge-connected graphs called the Frank number. 
This parameter can be used to refine a theorem by Nash-Williams~\cite{Na60}
stating that a graph has a $k$-arc-connected orientation if and only if it is $2k$-edge-connected.
\begin{definition}\label{def:fn}
For a $3$-edge-connected graph $G$, the \emph{Frank number} --~denoted by $fn(G)$~-- is the minimum number $k$ for which $G$ admits $k$ orientations such that every edge $e\in E(G)$ is deletable in at least one of them.
\end{definition}

Note that Definition~\ref{def:fn} does not make sense for graphs which are not $3$-edge-connected as such a graph $G$ has at least one edge which is not deletable in any orientation of $G$.

\medskip

A first general upper bound for the Frank number was established by H\"{o}rsch and Szigeti in~\cite{HS21}. They proved that $fn(G)\le 7$ for every 3-edge-connected graph $G$. Moreover, in the same paper it is shown that the Berge-Fulkerson conjecture~\cite{Se79} implies that $fn(G)\leq 5$. We improve the former result by showing the following upper bound. 

\begin{theorem}\label{thm:general}
Every $3$-edge-connected graph $G$ has $fn(G)\leq 4$.
\end{theorem}

We would also like to note that in~\cite{BB24} Bar{\'a}t and Bl{\'a}zsik  very recently independently proved that $fn(G)\leq 5$ using methods similar to ours based on our Lemma~\ref{lemma:jedna_dva} from this paper. 

\medskip

In~\cite{HS21}, H\"{o}rsch and Szigeti also conjectured that every $3$-edge-connected graph $G$ has $fn(G)\leq 3$ and showed that the Petersen graph has Frank number equal to $3$. In this paper we conjecture a stronger statement: 


\begin{conjecture}\label{conjecture1}
The Petersen graph is the only cyclically $4$-edge-connected
graph with Frank number greater than $2$. 
\end{conjecture}

Bar{\'a}t and Bl{\'a}zsik showed in~\cite{BB24_quest} that for any $3$-edge-connected graph $G$, there exists a $3$-edge-connected cubic graph $H$ with $fn(H)\geq fn(G)$. 
Corollary~\ref{cor:cubic_extension} in Section~\ref{subsec:reduction_cubic}
extends this result by showing that for any cyclically $4$-edge-connected $3$-edge-connected graph $G$, there exists a cyclically $4$-edge-connected cubic graph $H$ with $fn(H)\geq fn(G)$.
Hence, it is enough to prove the conjecture for cubic graphs and in the remainder, we will mainly focus on them.
Note that since cubic graphs cannot be $4$-edge-connected, their Frank number is at least $2$.

H\"{o}rsch and Szigeti proved in~\cite{HS21} that every $3$-edge-connected $3$-edge-colourable graph has Frank number at most $3$. We remark that such graphs are always cubic.
We strengthen this result by showing that these graphs have Frank number equal to $2$.
In fact, we prove the following more general theorem. 

\begin{theorem}\label{thm:4flow}
If $G$ is a $3$-edge-connected graph admitting a nowhere-zero $4$-flow, then $fn(G)\le2$. In particular, $fn(G)=2$ for every $3$-edge-connected $3$-edge-colourable graph~$G$.
\end{theorem}

It is also verified in~\cite{BB24_quest} that several well-known infinite families of 3-edge-connected graphs have Frank number 2. This includes wheel graphs, M\"obius ladders, prisms, flower snarks and an infinite subset of the generalised Petersen graphs. Note that except for the wheel graphs and flower snarks, these families all consist of $3$-edge-colourable graphs. In the same paper it is also conjectured that every $3$-edge-connected hamiltonian cubic graph has Frank number $2$. Since every hamiltonian cubic graph is $3$-edge-colourable, Theorem~\ref{thm:4flow} also proves this conjecture.

The main tool in the proofs of the two mentioned results make use of  nowhere-zero integer flows. We give a sufficient condition for an edge to be deletable in an orientation which is the underlying orientation of some all-positive nowhere-zero $k$-flow and construct two specific nowhere-zero $4$-flows that show that the Frank number is $2$.

Moreover, we also give two sufficient conditions for cyclically $4$-edge-connected cubic graphs to have Frank number $2$. We propose a heuristic algorithm and an exact algorithm for determining whether the Frank number of a $3$-edge-connected cubic graph is 2. The heuristic algorithm makes use of the sufficient conditions mentioned earlier.
Using our implementation of these algorithms we show that the Petersen graph is the only cyclically $4$-edge-connected cubic graph up to $36$ vertices with Frank number greater than $2$. This implies a 
positive answer for Conjecture~\ref{conjecture1} up to this order in the family of cubic graphs.

After the introduction and preliminaries, our paper is divided into two main sections: Section~\ref{sec:theoretical_results} which is devoted to theoretical results and Section~\ref{sec:algorithm} which focuses on the algorithmic aspects of this problem. More precisely, in Section~\ref{sec:theoretical_results} we first prove our key Lemma~\ref{lemma:jedna_dva} and use it to prove Theorems~\ref{thm:general} and \ref{thm:4flow}. Here, we also provide sufficient conditions for a cubic graph to have Frank number 2. In Section~\ref{sec:algorithm} we describe the algorithms and use them to check Conjecture~\ref{conjecture1} for nontrivial non-3-edge-colourable cubic graphs up to 36 vertices. Together with our theoretical results this proves that there is no cubic counterexample to Conjecture~\ref{conjecture1} up to 36 vertices. 

\subsection{Preliminaries}

Let $\mathcal{H}$ be an abelian group. An \emph{$\mathcal{H}$-flow $(o,f)$} on a graph $G$ consists of an orientation $(G,o)$ and a valuation $f:E(G)\rightarrow \mathcal{H}$ assigning elements of $\mathcal{H}$ to the edges of $G$ in such a way that for every vertex $v$ of $G$ the sum of the values on the incoming edges is the same as the sum of the values on the outgoing edges from $v$. A $\mathbb{Z}$-flow is called a \textit{$k$-flow} if the function $f$ only takes values in $\{0,\pm1,\pm2\ldots,\pm(k-1)\}$. 
An $\mathcal{H}$-flow $(o,f)$ (or a $k$-flow) is said to be \emph{nowhere-zero} if the value of $f$ is not the identity element $0\in \mathcal{H}$ ($0\in\mathbb{Z}$) for any edge of $E(G)$.



A nowhere-zero $k$-flow on $G$ is said to be \emph{all-positive} if the value $f(e)$ is positive for every edge $e$ of $G$. Every nowhere-zero $k$-flow can be transformed to an all-positive nowhere-zero $k$-flow by changing the orientation of the edges with negative $f(e)$ and changing negative values of $f(e)$ to $-f(e)$.


 Let $(G, o)$ be an orientation of a graph $G$. Let $H$ be a subgraph of $G$. If the context is clear we write $(H,o)$ to be the orientation of $H$ where $o$ is restricted to $H$. We define the set $D(G, o)\subseteq E(G)$ to be the set of all edges of $G$ which are deletable in $(G,o)$. Let $u, v\in V(G)$, if the edge $uv$ is oriented from $u$ to $v$, we write $u\rightarrow v$.

In the following proofs we will combine two flows into a new one as follows. Let $(o_1,f_1)$ and $(o_2,f_2)$ be $k$-flows on subgraphs $G_1$ and $G_2$ of a graph $G$, respectively. For $i\in\{1,2\}$ we extend the flow $(o_i,f_i)$ to flows on $G$, still called $(o_i,f_i)$, by setting the value of $f_i$ to be 0 and setting the orientation $o_i$ arbitrarily for the edges not in $G_i$. For those edges $e$ of $G$ where $o_1(e)\neq o_2(e)$, we change both the orientation of $o_2(e)$ and the value $f_2(e)$ to $-f_2(e)$ thereby transforming $(o_2,f_2)$ to a flow $(o_1, f_2')$. The \emph{combination} of flows $(o_1,f_1)$ and $(o_2,f_2)$ is the flow $(o_1,f_1+f_2')$ on $G$. Transforming this obtained flow to an all-positive flow, we get a flow $(o, f)$ on $G$, which we call the \emph{positive combination} of flows $(o_1,f_1)$ and $(o_2,f_2)$.

A \emph{smooth orientation} of a set of edge-disjoint circuits is an orientation such that for every circuit in the set, one edge is incoming and one edge is outgoing at every vertex in the circuit.

Let $G_1$ and $G_2$ be two subgraphs of $G$ and let $(G_1,o_1)$ and $(G_2,o_2)$ be two orientations. Let $Z\subseteq E(G_1)\cap E(G_2)$. We say that $(G_1,o_1)$ and $(G_2,o_2)$ are \emph{consistent} on $Z$ if $o_1$ and $o_2$ agree on all the edges from $Z$. 

\section{Theoretical results}\label{sec:theoretical_results}

Let $(o,f)$ be an all-positive nowhere-zero $k$-flow on a cubic graph $G$. An edge $e$ with $f(e)=2$ is called a \emph{strong $2$-edge} if $G$ has no 3-edge-cut containing the edge $e$ such that the remaining edges of the cut have value 1 in $f$.

\begin{lemma} \label{lemma:jedna_dva}
Let $G$ be a $3$-edge-connected graph and let $(o,f)$ be an all-positive nowhere-zero $k$-flow on $G$ for some integer $k$. Then all edges of $G$ which receive value 1 and all strong $2$-edges in $(o,f)$ are deletable in~$o$.
\end{lemma}
\begin{proof}
    Let $e$ be an edge with $f(e)=1$.
    Suppose that there exist two vertices of $G$, say $u$ and $v$, such that there is no oriented $uv$-path in $(G-e,o)$. Let $W$ be the set of vertices of $G$ to which there exists an oriented path from $u$ in $(G-e, o)$. Obviously, we have $u\in W$ and $v\not\in W$. Let $W'=V(G)-W$. Let us look at the edge-cut $S$ of $G-e$ between $W$ and $W'$. All the edges in $S$ must be oriented from $W'$ to $W$, otherwise for some vertex from $W'$ there would exist an oriented path from $u$ to this vertex. Moreover, as $G$ is $3$-edge-connected, we have $|S|\ge2$.
    
    Now consider the edge-cut $S^*$ between $W$ and $W'$ in $G$; either $S^*=S$ or $S^*=S\cup \{e\}$. Recall that $(o,f)$ is an all-positive nowhere-zero flow. Since $(o,f)$ is a flow, it holds that on any edge-cut the sum of the values on the edges oriented in one direction equals the sum of the values on the edges oriented in the other direction. Since all the edges of $S$ are oriented in the same direction and have non-zero value, it cannot happen that $S^*=S$. So it must be the case that $S^*=S\cup \{e\}$ and all the edges of $S$ are oriented in the same direction and $e$ is oriented in the opposite orientation. But $f(e)=1$ and since $|S|\ge2$ and all the values of $f$ on the edges of $S$ are positive, this cannot happen either. Therefore we conclude that for any two vertices $u$ and $v$ there exists an oriented $uv$-path in $(G-e, o)$ and so $e$ is deletable.
    
    Assume now that $e$ is a strong $2$-edge and suppose that $e$ is not deletable. We define the cuts $S$ in $G-e$ and $S^*$ in $G$ similarly as above. Since $f$ is all-positive and $G$ is 3-edge-connected, $S^*$ has to contain exactly three edges, $e$ and two edges in $S$ oriented oppositely from $e$ and valuated 1. But this is impossible since $e$ is a strong 2-edge. Therefore $e$ is deletable. 
\end{proof}

\subsection{The Frank number of graphs with a nowhere-zero 4-flow} \label{sect:3edgecol}

In this section, we prove Theorem~\ref{thm:4flow}.
In the proof we utilize Lemma~\ref{lemma:jedna_dva} and carefully apply the fact that every nowhere-zero $4$-flow can be expressed as a combination of two $2$-flows.

\medskip

\begin{proof}[Proof of Theorem~\ref{thm:4flow}]
Since $G$ admits a nowhere-zero $4$-flow, it also admits a nowhere-zero $(\mathbb{Z}_2\times\mathbb{Z}_2)$-flow by a famous result of Tutte~\cite{Tu54}. Let us denote by $A$ the set of edges with value $(0,1)$, by $B$ the set of edges with value $(1,0)$ and by $C$ the set of edges with value $(1,1)$ in this flow. Note that by the nowhere-zero property there are no edges with value $(0,0)$.

Consider the subgraph $G_1$ of $G$ induced by $A\cup C$. Since their edges all had a flow value with a $1$ in the first coordinate and Kirchhoff's law holds around every vertex in $G$, $G_1$ is \emph{Eulerian}, i.e.\ the degree of every vertex of $G_1$ is even. Therefore, $G_1$ consists of edge-disjoint circuits. Note that a vertex can belong to more than one circuit. Similarly, the subgraph $G_2$ induced by $B\cup C$ is Eulerian and so consists of edge-disjoint circuits.

Now fix a smooth orientation $(G_1, o_1)$ of the circuits in $G_1$ and a smooth orientation $(G_2, o_2)$ of the circuits in $G_2$. Set the value $f_i$ to be $i$ for the edges lying in $G_i$. Denote by $(o,f)$ the positive combination of the flows $(o_1, f_1)$ and $(o_2,f_2)$. The value 1 in $(o,f)$ is on all the edges of $A$ and on those edges of $C$ that have different orientation in $(G_1,o_1)$ and $(G_2,o_2)$.

Now we construct a complementary all-positive nowhere-zero 4-flow on $G$ in a sense that this flow will have value 1 exactly on those edges where $(o,f)$ had not. 
For $i\in\{1,2\}$ we set $o_i'=o_i$ on $G_i$. We set $f_1'(e)=2$ if $e\in G_1$ and $f_2'(e)=-1$ if $e\in G_2$. We create a flow $(o',f')$ of $G$ as the positive combination of the flows $(o_1', f_1')$ and $(o_2',f_2')$.

Summing up, we have constructed two all-positive nowhere-zero 4-flows on $G$, namely flows $(o,f)$ and $(o',f')$. The edges of $A$ are valuated 1 in $(o,f)$. The edges of $B$ are valuated 1 in $(o',f')$. The edges $e$ of $C$ where $o_1(e) \neq o_2(e)$ are valuated 1 in $(o,f)$. The edges $e$ of $C$ where $o_1(e) = o_2(e)$ are valuated 1 in $(o',f')$. Therefore, each edge has value 1 either in $(o,f)$ or in $(o',f')$ and by Lemma~\ref{lemma:jedna_dva} we have that $fn(G)=2$. 

It is known that a cubic graph is 3-edge-colourable if and only if it admits a nowhere-zero 4-flow. Therefore the second part of the theorem follows.
\end{proof}

Since every hamiltonian cubic graph is $3$-edge-colourable, we have also shown the following conjecture by Bar{\'a}t and Bl{\'a}zsik~\cite{BB24_quest}.
\begin{corollary}
    If $G$ is a $3$-edge-connected cubic graph admitting a hamiltonian cycle, then $fn(G)=2$.
\end{corollary}

\subsection{A general upper bound for the Frank number}\label{sect:general_upper_bound}
In this section we prove Theorem~\ref{thm:general} and thereby improve the previous general upper bound by H\"orsch and Szigeti from 7 to 4. We note that in~\cite{BB24} Bar{\'a}t and Bl{\'a}zsik recently independently proved that $fn(G)\leq 5$. As a main tool we again use Lemma~\ref{lemma:jedna_dva}.


\medskip

\textit{Proof of Theorem~\ref{thm:general}.\ }
    By a result of Bar{\'a}t and Bl{\'a}zsik it is sufficient to prove the theorem for $3$-edge-connected cubic graphs.

    Let $G$ be a $3$-edge-connected cubic graph. By Seymour's $6$-flow theorem~\cite{Se81}, $G$ has a nowhere-zero $(\mathbb{Z}_2\times \mathbb{Z}_3)$-flow $(o,f)$.
    The edges $e\in E(G)$ for which $f(e)$ is non-zero in the first coordinate induce a subgraph $D$. As every vertex in $G$ can either have none or two of such edges, $D$ is a set of vertex-disjoint circuits.
    The edges $e\in E(G)$ for which $f(e)$ is non-zero in the second coordinate induce a subgraph $H'$. As $H'$ admits a nowhere-zero $3$-flow there are no vertices of degree $1$ in $H'$, hence $H'$ consists of a set of vertex-disjoint circuits and a subdivision of a cubic graph $H$.
    The graph $H$ is bipartite since it is cubic and has a nowhere-zero $3$-flow.

    We will create four all-positive nowhere-zero $k$-flows on $G$ using the subgraphs $D$ and $H'$ such that each edge is valuated $1$ in at least one of the flows. Then Lemma~\ref{lemma:jedna_dva} will imply the result.

    Since $H$ is cubic and bipartite, by K\H{o}nig's line colouring theorem~\cite{Ko16}, it is $3$-edge-colourable. Hence, it admits a proper edge-colouring with colours $a,b,c$. Fixing such a colouring $\varphi$, we find a (not-necessarily proper) edge-colouring $\varphi'$ on $H'$ as follows. If an edge $e$ in $H$ corresponds to a path $P$ in $H'$, we colour all the edges of $P$ in $H'$ by $\varphi(e)$. All the edges lying on circuits in $H'$ will receive the colour $a$ in $\varphi'$. 
    Denote the set of edges with colour $a, b$, or $c$ in $\varphi'$ in $H'$ by $A, B$, or $C$, respectively.

    Now fix a smooth orientation $(D, o_D)$ of the circuits in $D$ and an orientation of $H$ by directing all edges from one partite set to the other. We can find an orientation $(H',o_{H'})$ of $H'$ as follows. Each oriented edge in $H$ will correspond to an oriented path in $H'$, oriented in the same direction. We take any smooth orientation of the circuits of $H'$.
    This fixes an orientation $(H',o_{H'})$ of $H'$.
 
    We now partition the edges of $G$ based on the orientations and colors of $D$ and $H'$. Denote by $D_0$ the set of edges which lie only in $D$. Denote by $A_0$ the set of edges which lie only in $A$, by $A_{+}$ the set of edges $e$ lying in $A$ and $D$ such that $o_{H'}$ and $o_D$ have the same direction for $e$ and by $A_{-}$ the set of edges $e$ in $A$ and $D$ such that $o_{H'}$ and $o_D$ direct $e$ oppositely. Similarly, we define $B_0$, $B_{+}$, $B_{-}$ and $C_0$, $C_{+}$ and $C_{-}$. It is easy to check that every edge belongs to exactly one of $D_0$, $A_0$, $A_{+}$, $A_{-}$, $B_0$, $B_{+}$, $B_{-}$, $C_0$, $C_{+}$, and $C_{-}$.

    Each of the four nowhere-zero flows $(o_i, h_i)$ for $i\in\{1,2,3,4\}$ will be the positive combination of a flow on $D$ and a flow on $H'$. In each of the four cases cases we proceed as follows. 
    
    We define flows $(o_{i,1}, g_{i,1})$ on $D$ for $i\in\{1,2,3,4\}$. For edges $e\in E(D)$, let $o_{i,1}(e) = o_D(e)$ and $g_{i,1}(e) = g_{i,D}$ where $g_{i,D}$ is the value according to Table~\ref{tab:4Flows}.
    
    We define flows $(o_{i,2}, g_{i,2})$ on $H'$ for $i\in\{1,2,3,4\}$. For edges $e\in E(H')$, let $o_{i,2}(e) = o_{H'}(e)$ and let $g_{i,2}(e)$ equal $g_{i,A}$, $g_{i,B}$ or $g_{i,C}$ if $e$ is in $A$, $B$ or $C$, respectively,  where $g_{i,A}$, $g_{i,B}$, and $g_{i,C}$ are three values that sum to $0$, according to Table~\ref{tab:4Flows}.
    
    The flow $(o_i, h_i)$ will be the positive combination of $(o_{i,1}, g_{i,1})$ and $(o_{i,2}, g_{i,2})$. For each $i\in\{1,2,3,4\}$ the flow values are given in Table~\ref{tab:4Flows}.



    Since for each $i\in\{1,2,3,4\}$, the sum of $g_{i,A}$, $g_{i, B}$ and $g_{i,C}$ is zero, it is easy to see that $(o_i, h_i)$ is a $k$-flow. Moreover, as $g_{i, D}$, $g_{i,A}$, $g_{i, B}$ and $g_{i,C}$ are non-zero and $\lvert g_{i, D}\lvert$ differs from $\lvert g_{i,A}\rvert$, $\lvert g_{i, B}\rvert$ or $\lvert g_{i,C}\rvert$, we see that each $(o_i, h_i)$ is nowhere-zero.

    Finally, one can see that for every edge $e$, $h_i(e)$ is $1$ for at least one $i$.
    The result follows. \hfill $\Box$

    \begin{table}[!htb]
        \centering
        \begin{tabular}{c | c c c c | c}
            Flow & $g_{i,D}$ & $g_{i,A}$ & $g_{i,B}$ & $g_{i,C}$ & $h_i(e) = 1$\\\hline
            $(o_1, h_1)$ & $1$ & $2$ & $2$ & $-4$ & $e\in D_0\cup {A_{-}} \cup {B_{-}}$\\
            $(o_2, h_2)$ & $3$ & $1$ & $1$ & $-2$ & $e\in A_0\cup B_0 \cup {C_{+}}$\\
            $(o_3, h_3)$ & $2$ & $3$ & $-4$ & $1$ & $e\in C_0\cup {A_{-}}\cup {C_{-}}$\\
            $(o_4, h_4)$ & $2$ & $-3$ & $-1$ & $4$ & $e\in {A_{+}}\cup {B_{+}}\cup B_0$
        \end{tabular}
        \caption{Four nowhere-zero $k$-flows defined by the procedure described in the proof of Theorem~\ref{thm:general}. Each $(o_i, h_i)$ is the positive combination of a flow on $D$ having value $g_{i,D}$ on each edge and a flow on $H'$, whose edge set can be partitioned into sets $A$, $B$ and $C$, having value $g_{i,A}$ on edges of $A$, $g_{i,B}$ on edges of $B$ and $g_{i,C}$ on edges of $C$. These values are found on the $i$'th row in their respective column. The final column indicates which edges obtain value $1$ in each of the positive combinations.}
        \label{tab:4Flows}
    \end{table}

\subsection{Reduction to cubic graphs}\label{subsec:reduction_cubic}
Barát and Blázsik showed in~\cite{BB24_quest} that for any $3$-edge-connected graph $G$, there exists a $3$-edge-connected cubic graph $H$ with $fn(H)\geq fn(G)$. Using the notion of \emph{local cubic modification}, i.e.\ replace a vertex $v$ of degree $d\geq 3$ with a circuit $v_1v_2\ldots v_d$ and replace each edge $vx_i$, where $x_1,\ldots, x_d$ are the neighbours of $v$ with an edge $v_jx_i$ such that every $v_i$ has degree $3$. Starting with a graph $G$ and applying this operation gives the local cubic modification $G_v$ of $G$ at $v$. Note that $G_v$ is not unique and depends on the perfect matching chosen between $v_1,\ldots, v_d$ and $x_1,\ldots, x_d$.

We extend this result to cyclically $4$-edge-connected $3$-edge-connected graphs.

\begin{lemma}
    For any cyclically $4$-edge-connected $3$-edge-connected graph $G$ and an arbitrary vertex $v\in V(G)$ of degree at least $4$. There exists a local cubic modification $G_v$ of $G$ at $v$ such that $G_v$ is cyclically $4$-edge-connected and $3$-edge-connected. 
\end{lemma}
\begin{proof}
    Fix a local cubic modification $G_v$ of $G$. If $G_v$ contains a bridge, then we have a bridge in $G$, so $G_v$ is $2$-edge-connected. Barát and Blázsik showed in~\cite[Lemma~4.1]{BB24_quest}
    that a $2$-edge-cut must intersect $C_v = v_1\ldots, v_d$ twice and hence $v$ is a cut vertex of $G$.
    
    Now suppose that $G_v$ contains a cyclic $3$-edge-cut $X$. 
    We show that in this case $X$ must also intersect the circuit $C_v = v_1\ldots, v_d$ exactly twice. A cut must intersect $C_v$ an even number of times, so suppose that $X$ does not intersect $C_v$. If $G_v - X$ has circuits which are not $C_v$ in each component, then $X$ is also a cyclic $3$-edge-cut of $G$, which is a contradiction. Hence, one of the components of $G_v - X$ can only have $C_v$ as a circuit. This corresponds to an acyclic component of $G - X$ in which all but the vertices incident with $X$ in $G$ have degree at least $3$, since $G$ is $3$-edge-connected. However, this can only happen if the acyclic component is the single vertex $v$, in contradiction with the fact that $v$ has degree at least $4$. Hence, a cyclic $3$-edge-cut in $G_v$ for any choice of perfect matching must intersect $C_v$ twice. We conclude that if $G_v$ contains a 2-edge cut or a cycle separating 3-edge cut $X$, the cut $X$ is intersected by $C_v$ in two edges.

    We see that $G - v$, which is not necessarily connected,
    consists of $2$-edge-connected components or single vertices which are connected to other components by bridges of $G - v$. We label all such components connected to at most one other via a bridge in $G-v$ by $K_1, \ldots, K_k$. Note that by $3$-edge-connectivity for any $K_i$, we have $\lvert V(K_i)\cap N_G(v)\rvert\geq 2$.


    We now construct a perfect matching such that no two or three edges of $G_v$ define a $2$-edge-cut or a cyclic $3$-edge-cut of $G_v$, respectively. By the previous remark, we have $d\geq 2k$. For $i\in \{1,\ldots, k\}$, connect $v_i$ with a vertex of $K_i$, for $i\in \{k+1, \ldots, 2k\}$, connect $v_i$ with a vertex of $K_{i-k}$, connect the remaining vertices of $C_v$ arbitrarily to the remaining neighbours of $v$ in $G$. 
    Let $X$ be a $2$-edge-cut or a cyclic $3$-edge-cut of $G_v$.
    Then it must separate a vertex $x$ in some $K_{i_1}$ from a vertex $y$ in some $K_{i_2}$, with $i_1 < i_2$. Clearly, $x$ is connected to $v_{i_1}$, so one edge of $X$ is $v_{i'_1}v_{i'_1+1}$ with $i_1 \leq i'_1 < i'_1+1 \leq i_2$. Otherwise, $x$ and $y$ are are in the same component of $G_v-X$. Similarly, the other edge of $X$ must be $v_{i'_2}v_{i'_2+1}$ with $i_2\leq i'_2<i'_2+1 \leq k+i_1$. However, $v_{k+i_1}$ is still connected to $v_{k+i_2}$ on $C_v$, hence $x$ and $y$ are connected. 
\end{proof}

Using Barát and Blázsik's Lemma~4.3 from~\cite{BB24_quest}, we obtain the following corollary.

\begin{corollary}\label{cor:cubic_extension}
    Let $G$ be a cyclically $4$-edge-connected $3$-edge-connected graph, then there exists a cyclically $4$-edge-connected cubic graph $H$ with $fn(H)\geq fn(G)$.
\end{corollary}

\subsection{Sufficient conditions for Frank number 2}\label{sect:conditions_fn2}

The following lemmas and theorems give two sufficient conditions for a cyclically $4$-edge-connected cubic graph to have Frank number $2$. These will be used in the algorithm in Section~\ref{sec:algorithm}. We note that for the conditions to hold the graphs need to have a $2$-factor with exactly two odd circuits. For the graphs we consider in Section~\ref{sec:algorithm}, the vast majority will have such a $2$-factor. Even though the ideas we use can possibly be extended to $2$-factors with a higher number of odd circuits, the proofs will be more involved and they will yield little speedup for our computations.

In the proof of the next lemma we will use the following proposition which can be found in \cite{BB24_quest} as Proposition 2.2.

\begin{proposition} \label{prop:uv}
Let $(G,o)$ be a strong orientation of a graph $G$. Assume that an edge $e=uv$ is oriented from u to v in $(G,o)$. The edge $e$ is deletable in $(G, o)$ if and only if there exists an oriented $uv$-path in $(G-e,o)$. 
\end{proposition}

\begin{lemma} \label{lemma:oo}
  Let $(G,o)$ be a strong orientation of a cubic graph $G$. Let $e_1=u_1v_1$ and $e_2=u_2v_2$ be two non-adjacent edges in $G$ such that $(G,o)$ contains $u_1\rightarrow v_1$ and $u_2 \rightarrow v_2$. Assume that both $e_1$ and $e_2$ are deletable in $(G,o)$. Create a cubic graph $G'$ from $G$ by subdividing the edges $e_1$ and $e_2$ with vertices $x_1$ and $x_2$, respectively, and adding a new edge between $x_1$ and $x_2$. Let $(G',o')$ be the orientation of $G'$ containing $u_1\rightarrow x_1$, $x_1\rightarrow v_1$, $x_1\rightarrow x_2$, $u_2\rightarrow x_2$, $x_2\rightarrow v_2$ and such that $o'(e)=o(e)$ for all the remaining edges of $G'$. Then $$D(G',o')\supseteq (D(G,o)-\{e_1,e_2\})\cup\{x_1v_1,x_1x_2, u_2x_2\}.$$
\end{lemma}
\begin{proof}
First of all, define $(G',o')$ as in the statement of this Lemma. Then it is a strong orientation. Indeed, since $(G,o)$ is a strong orientation of $G$, any edge-cut in $G$ contains edges in both directions in $(G,o)$. Therefore, any edge-cut in $G'$ contains edges in both directions in $(G',o')$.

Now we are going to show that $$D(G',o')\supseteq (D(G,o)-\{e_1,e_2\})\cup\{x_1v_1,x_1x_2, u_2x_2\}.$$
Let $e=pr\in D(G,o)-\{e_1,e_2\}$ and let $p\rightarrow r$ be an oriented edge in $(G,o)$. By Proposition~\ref{prop:uv} we have to show that there is an oriented $pr$-path in $(G'-e,o')$. If neither of $p$ and $r$ belongs to $\{x_1, x_2\}$, that is $p$ and $r$ belong to $V(G)$, then, since $(G,o)$ is a strong orientation of $G$, there exists an oriented $pr$-path $R$ in $(G,o)$. Then $R$ with possible subdivisions by $x_1$ and $x_2$ if $e_1\in R$ or $e_2\in R$ is the required $pr$-path in $(G'-e,o')$.

If $pr=x_1v_1$, then we find a $v_2v_1$-path $R_1$ in $(G-e_1,o)$, which exists since $e_1$ is deletable in $o$. The path $x_1x_2v_2R_1$ is an oriented $x_1v_1$-path in $(G'-x_1v_1,o')$.

If $pr=u_2x_2$, then we find a $u_2u_1$-path $R_2$ in $(G-e_2,o)$, which exists since $e_2$ is deletable in $o$. The path $R_2u_1x_1x_2$ is an oriented $u_2x_2$-path in $(G'-u_2x_2,o')$.

Finally, if $pr=x_1x_2$, we find a $v_1u_2$-path $R_3$ in $(G,o)$. The path $x_1v_1R_3u_2x_2$ is an oriented $x_1x_2$-path in $(G'-x_1x_2,o')$.
\end{proof}



\begin{theorem}\label{thm:2OddCycles} Let $G$ be a cyclically $4$-edge-connected cubic graph. Let $C$ be a $2$-factor of $G$ with exactly two odd circuits, say $N_1$ and $N_2$ (and possibly some even circuits). Let $e=x_1x_2$ be an edge of $G$ such that $x_1\in V(N_1)$ and $x_2\in V(N_2)$. Let $F=G-C$ and let $M$ be a maximum matching in $C-\{x_1,x_2\}$. If there exists a smooth orientation of the circuits in $F-\{e\}\cup M$ and a smooth orientation of $C$ such that each is consistent on the edges of $N_i$ at distance 1 from $x_i$ for both $i\in\{1,2\}$, see Fig.~\ref{fig:oo}, then $fn(G)=2$.
\end{theorem}
\begin{proof}
For $i\in\{1,2\}$ denote by $u_i$ and $v_i$ the vertices of $N_i$ which are adjacent to $x_i$ and let the rest of notation be as in the statement of the theorem. Consider the graph $G\sim e$, that is the graph created by deleting $e$ and \emph{smoothing} the two 2-valent vertices, i.e.\ removing them and adding an edge between their neighbours. First, we observe that $G\sim e$ is cyclically 3-edge-connected. This can be easily seen, because if $G\sim e$ had a cycle-separating $k$-edge-cut $S$ for some $k<3$, then $S\cup \{e\}$ would be a set of $k+1$ edges separating two components containing circuits, so the cyclic connectivity of $G$ would be at most $k+1<4$, a contradiction.

Let $(F-\{e\}\cup M,o_2)$ be a smooth orientation of the circuits in $F-\{e\}\cup M$ consistent on the edges of $N_i$ which are at distance 1 from $x_i$ for both $i\in\{1,2\}$. Let $(C,o_1)$ be such a smooth orientation of the circuits in $C$ that the edges of $N_i\cap M$ incident with $u_i$ and $v_i$ on $N_i$ are oriented equally in $(C,o_1)$ and $(F-\{e\}\cup M,o_2)$. By our assumption these orientation exists. With slight abuse of notation, we will also consider $C$ and $F-\{e\}\cup M$ to be subgraphs of $G\sim e$ and consider $(C,o_1)$ and $(F-\{e\}\cup M,o_2)$ to be orientations of these subgraphs.

We define two nowhere-zero 4-flows $(o',f')$ and $(o'',f'')$ on $G\sim e$ such that $D(G \sim e, o')\cup D(G \sim e, o'') = E(G \sim e)$ and $\{u_1v_1, u_2v_2\}\subset D(G\sim e, o')\cap D(G\sim e, o'')$. This will by Lemma~\ref{lemma:oo} imply that $fn(G)=2$.

We define flows $(o_1, f_1')$ on $C$ and $(o_2, f_2')$ on $F-\{e\}\cup M$. For edges $d\in E(C)$, let $f_1'(d) = 1$. For edges $d\in E(F-\{e\}\cup M)$, let $f_2'(d) = -2$.
The flow $(o',f')$ on $G\sim e$ will be the positive combination of $(o_1, f_1')$ and $(o_2, f_2')$.

Similarly, we define flows $(o_1, f''_1)$ on $C$ and $(o_2, f''_2)$ on $F-\{e\}\cup M$. For edges $d\in E(C)$, let $f_1''(d) = 2$. For edges $d\in E(F-e\cup M)$, let $f_2''(d) = 1$.
The flow $(o'', f'')$ on $G\sim e$ will be the positive combination of $(o_1, f''_1)$ and $(o_2, f''_2)$.

It is easy to see that both $(o', f')$ and $(o'', f'')$ are nowhere-zero. We also see that the edges of $C-M$ are valuated 1 in $(o',f')$, the edges of $F-\{e\}$ are valuated 1 in $(o'',f'')$ and the edges of $M$ are valuated 1 in $(o',f')$ if $o_1$ and $o_2$ agree and are valuated $1$ in $(o'',f'')$ if $o_1$ and $o_2$ disagree. Therefore, by Lemma~\ref{lemma:jedna_dva}, all edges of $G\sim e$ are deletable in one of these orientations and so $D(G\sim e, o')\cup D(G\sim e, o'') = E(G\sim e)$.

It remains to show that the edges $u_1v_1$ and $u_2v_2$ are deletable both in $(G\sim e,o')$ and in $(G\sim e,o'')$. By Lemma~\ref{lemma:jedna_dva}, they are deletable in $(G\sim e,o')$. Now we show that $u_1v_1$ and $u_2v_2$ are strong 2-edges in $(o'',f'')$. Suppose that, for some $i\in\{1,2\}$, the edge $u_iv_i$ is not a strong 2-edge in $(o'',f'')$, say $u_1v_1$ is not a strong 2-edge. Since we have already observed that $G\sim e$ is cyclically 3-edge-connected, this implies that $u_1v_1$ belongs to a 3-edge-cut, say $R$, and the other two edges of this 3-edge-cut have to be valuated 1. The cut $R$ must be cycle-separating. Otherwise, it would separate either $u_1$ or $v_1$ from the rest of the graph, contradicting the fact that the edges of $M$ incident with $u_1$ and $v_1$ are valuated 3 in $(o'',f'')$.

We will use the symbols $N_1$ and $N_2$ to denote also the circuits in $G\sim e$ corresponding to $N_1$ and $N_2$. Since every edge-cut intersects a circuit an even number of times, $R$ contains two edges from $N_1$
(one of them is $u_1v_1$, let the other be $g$) and an edge $f\in E(F)\setminus \{e\}$. Let $(V_1, V_2)$ be the partition of $V(G)$ corresponding to the cut $R$. All the vertices of the circuit $N_2$ belong either to $V_1$ or to $V_2$. In the former case, the partition $(V_1\cup\{x_1,x_2\},V_2)$ and in the latter case, the partition $(V_1,V_2\cup\{x_1,x_2\})$ form a partition corresponding to a 3-edge-cut in $G$. As the two edges of this cut belonging to $N_1$ are independent the cut is a cycle-separating 3-edge-cut in $G$, which is a contradiction. 
\end{proof}

\begin{figure}[!htb]
 \centering
\def\circledarrow#1#2#3{ 
	\draw[#1,<-] (#2) +(80:#3) arc(80:-260:#3);
}
\begin{tikzpicture}[label distance=-1mm, scale=0.8]
	\path[use as bounding box] (-2,-2) rectangle (2,2.2);
	\node [circle,fill,scale=0.25, label=above right:$x_1$] (x1) at (-1,0) {};
	\node [circle,fill,scale=0.25, label=above left:$x_2$] (x2) at (1,0) {};
	\node [circle,fill,scale=0.25, label=below right:$u_1$] (u1) at (-1.2,1) {};
	\node [circle,fill,scale=0.25, label=below left:$u_2$] (u2) at (1.2,1) {};
	\node [circle,fill,scale=0.25, label=above right:$v_1$] (v1) at (-1.2,-1) {};
	\node [circle,fill,scale=0.25, label=above left:$v_2$] (v2) at (1.2,-1) {};
	\node [circle,fill,scale=0.25] (z1) at (-2,2) {};
	\node [circle,fill,scale=0.25] (z2) at (2,2) {};
	\node [circle,fill,scale=0.25] (y1) at (-2,-2) {};
	\node [circle,fill,scale=0.25] (y2) at (2,-2) {};

	\draw (x1) to (x2);
 	\draw [<-, thick] (z1) to (u1);
	\draw (u1) to (x1) to (v1);
 	\draw [<-, thick] (v1) to (y1);
    \draw [->, thick] (z2) to (u2);
	\draw (u2) to (x2) to (v2);
    \draw [->, thick] (v2) to (y2);

	\path (x1) ++(-1.5,0) node {$N_1$};
	\circledarrow{thick}{-2.5,0}{1};
	\path (x2) ++(1.5,0) node {$N_2$};
	\circledarrow{thick}{2.5,0}{1};

	\path [draw, <-, thick] (u1) -- ++(0.5,0.2);
	\path [draw, ->, thick] (z1) -- ++(0.5,0.4);
	\path [draw, ->, thick] (v1) -- ++(0.5,-0.2);
	\path [draw, <-, thick] (y1) -- ++(0.5,-0.4);
	\path [draw, ->, thick] (u2) -- ++(-0.5,0.2);
	\path [draw, <-, thick] (z2) -- ++(-0.5,0.4);
	\path [draw, <-, thick] (v2) -- ++(-0.5,-0.2);
	\path [draw, ->, thick] (y2) -- ++(-0.5,-0.4);

	\draw (y1) to[bend left,looseness=2,in=70,out=110] (z1);
	\draw (y2) to[bend right,looseness=2,in=-70,out=-110] (z2);
\end{tikzpicture}
 \caption{A smooth orientation of the circuits in $F-\{x_1x_2\}\cup M$ and of those in $C$ such that each is consistent on the edges of $N_i$ at distance $1$ from $x_i$, for $i\in\{1,2\}$.} 
\label{fig:oo}
\end{figure}
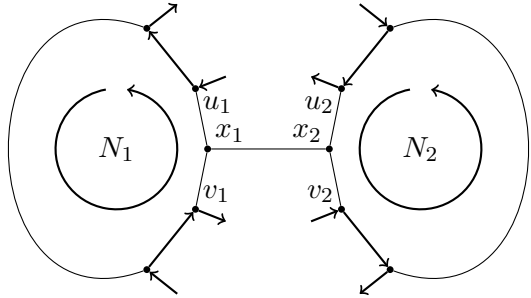

\begin{lemma} \label{lemma:oeo}
Let $(G,o)$ be a strong orientation of a cubic graph $G$. Let $e_1=u_1v_1$, $e_2=u_2v_2$, and $f=w_2w_1$ be pairwise independent edges in $G$ such that $(G,o)$ contains $u_1\rightarrow v_1$, $u_2\rightarrow v_2$, and $w_2\rightarrow w_1$ and such that these edges are deletable in $(G, o)$. Let a cubic graph $G'$ be created from $G$ by performing the following steps:
\begin{itemize}
    \item subdivide the edges $e_1$ and $e_2$ with the vertices $x_1$ and $x_2$, respectively, 
    \item subdivide the edge $w_1w_2$ with the vertices $y_1$ and $y_2$ (in this order), and
    \item add the edges $x_1y_1$ and $x_2y_2$.
\end{itemize}
Let $(G',o')$ be the orientation of $G'$ containing $u_1\rightarrow x_1$, $x_1\rightarrow v_1$, $y_1\rightarrow w_1$, $y_2\rightarrow y_1$, $w_2\rightarrow y_2$,  $u_2\rightarrow x_2$, $x_2\rightarrow v_2$ and such that $o'(e)=o(e)$ for all the remaining edges of $G'$ except for $x_1y_1$ and $x_2y_2$. Then
\begin{enumerate}[(a)]
\item if $(G',o')$ contains $y_1\rightarrow x_1$ and $x_2\rightarrow y_2$, $(G',o')$ is a strong orientation of $G'$ and 
        $D(G',o')\supseteq (D(G,o)-\{e_1,e_2,f\})\cup \{u_1x_1,x_1y_1,y_1w_1,y_2w_2, x_2y_2,x_2v_2\}$ (Fig.~{\rm \ref{fig:oeo}(left)});
\item if $(G',o')$ contains $x_1\rightarrow y_1$ and $y_2\rightarrow x_2$, $(G',o')$ is a strong orientation of $G'$ and 
        $D(G',o')\supseteq (D(G,o)-\{e_1,e_2,f\})\cup \{x_1v_1,y_1y_2,u_2x_2\}$ (Fig.~{\rm \ref{fig:oeo}(right)}).
\end{enumerate}
\end{lemma}
\begin{figure}
    \centering
    \begin{minipage}{0.49\linewidth}
        \centering
        \begin{tikzpicture}[label distance=-0.4mm, scale=0.7]
    \node [circle,fill,scale=0.25, label=right:$x_1$] (x1) at (-2,0) {};
	\node [circle,fill,scale=0.25, label=left:$x_2$] (x2) at (2,0) {};
	\node [circle,fill,scale=0.25, label=left:$u_1$] (u1) at (-2.2,1) {};
	\node [circle,fill,scale=0.25, label=right:$u_2$] (u2) at (2.2,1) {};
	\node [circle,fill,scale=0.25, label=left:$v_1$] (v1) at (-2.2,-1) {};
	\node [circle,fill,scale=0.25, label=right:$v_2$] (v2) at (2.2,-1) {};
	\node (ez1) at (-3,2) {};
	\node (ez2) at (3,2) {};
	\node (ey1) at (-3,-2) {};
	\node (ey2) at (3,-2) {};

	\node [circle,fill,scale=0.25, label=below:$y_1$] (y1) at (-0.5,-2) {};
	\node [circle,fill,scale=0.25, label=below:$y_2$] (y2) at (0.5,-2) {};
	\node [circle,fill,scale=0.25, label=below:$w_1$] (w1) at (-1.5,-2.2) {};
	\node [circle,fill,scale=0.25, label=below:$w_2$] (w2) at (1.5,-2.2) {};
	\node (ew1) at (-2.5,-2.8) {};
	\node (ew2) at (2.5,-2.8) {};

	\draw [<-, very thick, blue] (x1) to (y1);
	\draw [<-, very thick, blue] (y2) to (x2);
	\draw (ez1) to (u1);
	\draw [->, very thick, color=blue] (u1) to (x1);
	\draw [->, thick] (x1) to (v1);
	\draw (v1) to (ey1);
	\draw (ez2) to (u2);
	\draw [->, thick] (u2) to (x2);
	\draw [->, very thick, blue] (x2) to (v2);
	\draw (v2) to (ey2);

	\draw (ew1) to (w1);
	\draw [<-, very thick, blue] (w1) to (y1);
	\draw [<-, thick] (y1) to (y2);
	\draw [<-, very thick, blue] (y2) to (w2);
	\draw (w2) to (ew2);

	\path (x1) ++(-1.5,0) node {$N_1$};
	\path (x2) ++(1.5,0) node {$N_2$};
	\path (0,-1.5) ++(0,-1.5) node {$W$};


\end{tikzpicture}
    \end{minipage}
    \begin{minipage}{0.49\linewidth}
        \centering
        \begin{tikzpicture}[label distance=-0.4mm, scale=0.7]
	\node [circle,fill,scale=0.25, label=right:$x_1$] (x1) at (-2,0) {};
	\node [circle,fill,scale=0.25, label=left:$x_2$] (x2) at (2,0) {};
	\node [circle,fill,scale=0.25, label=left:$u_1$] (u1) at (-2.2,1) {};
	\node [circle,fill,scale=0.25, label=right:$u_2$] (u2) at (2.2,1) {};
	\node [circle,fill,scale=0.25, label=left:$v_1$] (v1) at (-2.2,-1) {};
	\node [circle,fill,scale=0.25, label=right:$v_2$] (v2) at (2.2,-1) {};
	\node (ez1) at (-3,2) {};
	\node (ez2) at (3,2) {};
	\node (ey1) at (-3,-2) {};
	\node (ey2) at (3,-2) {};

	\node [circle,fill,scale=0.25, label=below:$y_1$] (y1) at (-0.5,-2) {};
	\node [circle,fill,scale=0.25, label=below:$y_2$] (y2) at (0.5,-2) {};
	\node [circle,fill,scale=0.25, label=below:$w_1$] (w1) at (-1.5,-2.2) {};
	\node [circle,fill,scale=0.25, label=below:$w_2$] (w2) at (1.5,-2.2) {};
	\node (ew1) at (-2.5,-2.8) {};
	\node (ew2) at (2.5,-2.8) {};

	\draw [->, thick] (x1) to (y1);
	\draw [->, thick] (y2) to (x2);
	\draw (ez1) to (u1);
	\draw [->, thick] (u1) to (x1);
	\draw [->, very thick, blue] (x1) to (v1);
	\draw (v1) to (ey1);
	\draw (ez2) to (u2);
	\draw [->, very thick, blue] (u2) to (x2);
	\draw [->, thick] (x2) to (v2);
	\draw (v2) to (ey2);

	\draw (ew1) to (w1);
	\draw [<-, thick] (w1) to (y1);
	\draw [<-, very thick, blue] (y1) to (y2);
	\draw [<-, thick] (y2) to (w2);
	\draw (w2) to (ew2);

	\path (x1) ++(-1.5,0) node {$N_1$};
	\path (x2) ++(1.5,0) node {$N_2$};
	\path (0,-1.5) ++(0,-1.5) node {$W$};


\end{tikzpicture}
    \end{minipage}
    \caption{A part of $G'$ and orientation $(G',o')$ as defined in Lemma~\ref{lemma:oeo}. The left-hand-side corresponds with the orientation of (a) and the right-hand-side corresponds with the orientation of (b). If the conditions of Lemma~\ref{lemma:oeo} are met the thick, blue edges will be deletable.}
    \label{fig:oeo}
\end{figure}
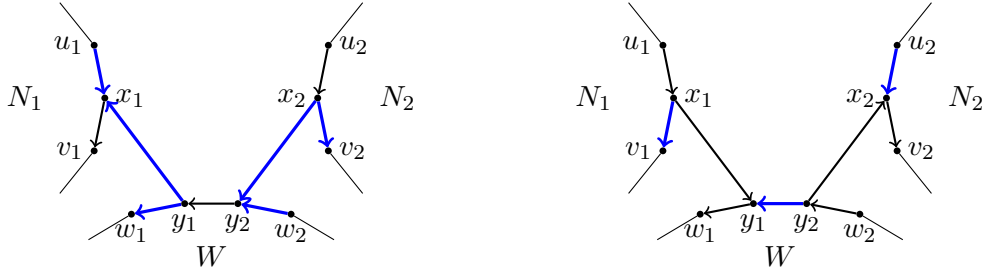
\begin{proof}
First of all, if $(G',o')$ is defined either as in (a) or as in (b), then it is a strong orientation. Indeed, since $(G,o)$ is a strong orientation of $G$, any edge-cut in $G$ contains edges in both directions in $(G,o)$. Therefore, any edge-cut in $G'$ contains edges in both directions in $(G',o')$.

To prove (a) we need to show that every edge from 
\[(D(G,o)-\{e_1,e_2,f\})\cup \{u_1x_1,x_1y_1,y_1w_1,y_2w_2, x_2y_2,x_2v_2\}\] is deletable in $(G',o')$. To do so, it is enough to show that for any edge $e=pr$ in this set oriented as $p\rightarrow r$, there exists an oriented $pr$-path in $(G'-e,o')$. If $e\in D(G,o)-\{e_1,e_2,f\}$, then we can use the fact that $e$ is a deletable edge in $(G,o)$ and therefore there is an oriented $pr$-path in $(G-e,o)$. The corresponding $pr$-path in $(G',o')$ possibly subdivided by some of the vertices $x_1,x_2,y_1,y_2$ is a required path.

If $pr=u_1x_1$, we take a $u_1w_2$-path $R_1$ in $G-e_1$ (it exists since $e_1$ is deletable in $(G, o)$). Possibly subdivide $R_1$ to $R_1'$ in $(G', o')$. The path $R_1'w_2y_2y_1x_1$ is an oriented $u_1x_1$-path in $(G'-u_1x_1,o')$, so $u_1x_1$ is deletable in $(G',o')$.

We proceed in this way. For every of the five remaining edges we find an oriented path $R$ in $(G,o)$ in which possibly one of the edges $e_1,e_2,f$ is forbidden (such a path exists as each of these three edges is deletable in $(G,o)$), possibly subdivide $R$ to $R'$ in $(G',o')$ and finally find the required oriented path $T$ in $(G'-e,o')$. It is summarised in the following table.

\begin{center}
\begin{tabular}{| c | c | c |}\hline
$e$ & $R$ & $T$\\\hline\hline
$u_1x_1$ & $u_1w_2$-path in $G - e_1$ & $R'w_2y_2y_1x_1$\\\hline
$y_1x_1$ & $w_1u_1$-path in $G$ & $y_1w_1R'u_1x_1$\\\hline
$y_1w_1$ & $v_1w_1$-path in $G-f$ & $y_1x_1v_1R'$\\\hline 
$w_2y_2$ & $w_2u_2$-path in $G-f$ & $R'u_2x_2y_2$\\\hline 
$x_2y_2$ & $v_2w_2$-path in $G$ & $x_2v_2R'w_2y_2$\\\hline
$x_2v_2$ & $w_1v_2$-path in $G-e_2$ & $x_2y_2y_1w_1R'$\\\hline 
\end{tabular}
\end{center}

To prove (b) we proceed analogously as in (a). There are only three edges we have to pay special attention to. Each of the edges corresponds to a line of the following table. This completes the proof.

\begin{center}
\begin{tabular}{| c | c | c |}\hline
$e$ & $R$ & $T$\\\hline\hline
$x_1v_1$ & $w_1v_1$-path in $G-e_1$ & $x_1y_1w_1R'$\\\hline
$y_2y_1$ & $v_2u_1$-path in $G-f$ & $y_2x_2v_2R'u_1x_1y_1$\\\hline
$u_2x_2$ & $u_2w_2$-path in $G-e_2$ & $R'w_2y_2x_2$\\\hline
\end{tabular}
\end{center}
\end{proof}


\begin{theorem}\label{thm:2Odd1Even}
Let $G$ be a cyclically $4$-edge-connected cubic graph with a $2$-factor $C$ containing precisely two odd circuits $N_1$ and $N_2$ and at least one even circuit $W$. 
Let $x_1y_1$, $y_1y_2$ and $y_2x_2$ be edges of $G$ such that $x_1\in V(N_1)$, $x_2\in V(N_2)$ and $y_1, y_2\in V(W)$. For $i\in\{1,2\}$ denote by $u_i$ and $v_i$ the vertices of $N_i$ which are adjacent to $x_i$ and by $w_i$ the the vertex in $W -\{y_1,y_2\}$ which is adjacent to $y_i$ in $W$. Let $F = G - C$ and let $M$ be a maximum matching in $C - \{x_1, y_1, y_2, x_2\}$. If there exists a smooth orientation of the circuits in $F - \{x_1y_1,x_2y_2\} \cup M$ and a smooth orientation of $C$ such that each is consistent on the edges of $N_i$ at distance 1 from $x_i$ for both $i\in\{1,2\}$ and on the edges of $W$ at distance 1 from $y_i$ but not incident with $y_{3-i}$, for $i\in\{1,2\}$, see Fig.~\ref{fig:2foeo},
and $G\sim x_1y_1\sim x_2y_2$ is cyclically $3$-edge-connected and has no cycle-separating $3$-edge-cut $\{e_1, e_2, e_3\}$ with $e_1\in \{u_1v_1, u_2v_2, w_1w_2\}$ and $e_2, e_3\in E(F - \{x_1y_1,x_2y_2\}\cup M)$, then $fn(G)= 2$.
\end{theorem}

\begin{proof}
Let $f_1=x_1y_1$, $f_2=x_2y_2$ and $G' := G\sim f_1\sim f_2$.
We define nowhere-zero 4-flows $(o',f')$ and $(o'',f'')$ on $G'$ similarly as in the proof of Theorem~\ref{thm:2OddCycles}. We can use the arguments from the latter theorem. Since we assume that none of the edges $u_1v_1, u_2v_2, w_1w_2$ are in a cycle-separating 3-edge-cut with two edges in $F - \{f_1,f_2\}\cup M$ they are strong $2$-edges in $(o'',f'')$
and the result follows from Lemma~\ref{lemma:jedna_dva} and Lemma~\ref{lemma:oeo}.
\end{proof}

\begin{figure}[!htb]
 \centering
\def\circledarrow#1#2#3{ 
	\draw[#1,<-] (#2) +(80:#3) arc(80:-260:#3);
}
\begin{tikzpicture}[label distance=-0.4mm, scale=0.8]
	\path[use as bounding box] (-5,-5.3) rectangle (5,2.4);
	\node [circle,fill,scale=0.25, label=right:$x_1$] (x1) at (-2,0) {};
	\node [circle,fill,scale=0.25, label=left:$x_2$] (x2) at (2,0) {};
	\node [circle,fill,scale=0.25, label=left:$u_1$] (u1) at (-2.2,1) {};
	\node [circle,fill,scale=0.25, label=right:$u_2$] (u2) at (2.2,1) {};
	\node [circle,fill,scale=0.25, label=left:$v_1$] (v1) at (-2.2,-1) {};
	\node [circle,fill,scale=0.25, label=right:$v_2$] (v2) at (2.2,-1) {};
	\node [circle,fill,scale=0.25] (ez1) at (-3,2) {};
	\node [circle,fill,scale=0.25] (ez2) at (3,2) {};
	\node [circle,fill,scale=0.25] (ey1) at (-3,-2) {};
	\node [circle,fill,scale=0.25] (ey2) at (3,-2) {};

	\node [circle,fill,scale=0.25, label=below:$y_1$] (y1) at (-0.5,-3) {};
	\node [circle,fill,scale=0.25, label=below:$y_2$] (y2) at (0.5,-3) {};
	\node [circle,fill,scale=0.25, label=below:$w_1$] (w1) at (-1.5,-3.2) {};
	\node [circle,fill,scale=0.25, label=below:$w_2$] (w2) at (1.5,-3.2) {};
	\node [circle,fill,scale=0.25] (ew1) at (-2.5,-4) {};
	\node [circle,fill,scale=0.25] (ew2) at (2.5,-4) {};

	\draw (x1) to (y1) to (y2) to (x2);
    \draw [<-, thick] (ez1) to (u1);
	\draw (u1) to (x1) to (v1);
 	\draw [<-, thick] (v1) to (ey1);
    \draw [->, thick] (ez2) to (u2);
	\draw (u2) to (x2) to (v2);
 	\draw [->, thick] (v2) to (ey2);

 	\draw [<-, thick] (ew1) to (w1);
	\draw (w1) to (y1);
	\draw (y1) to (y2);
	\draw (y2) to (w2);
 	\draw [<-, thick] (w2) to (ew2);

	\path (x1) ++(-1.5,0) node {$N_1$};
	\circledarrow{thick}{-3.5,0}{1};
	\path (x2) ++(1.5,0) node {$N_2$};
	\circledarrow{thick}{3.5,0}{1}
	\path (0,-3) ++(0,-1.5) node {$W$};
	\circledarrow{thick}{0,-4.5}{1};

	\path [draw, <-, thick] (u1) -- ++(0.5,0.2);
	\path [draw, ->, thick] (ez1) -- ++(0.5,0.4);
	\path [draw, ->, thick] (v1) -- ++(0.5,-0.2);
	\path [draw, <-, thick] (ey1) -- ++(0.5,-0.4);
	\path [draw, ->, thick] (u2) -- ++(-0.5,0.2);
	\path [draw, <-, thick] (ez2) -- ++(-0.5,0.4);
	\path [draw, <-, thick] (v2) -- ++(-0.5,-0.2);
	\path [draw, ->, thick] (ey2) -- ++(-0.5,-0.4);
	\path [draw, <-, thick] (w1) -- ++(-0.2,0.5);
	\path [draw, ->, thick] (w2) -- ++(0.2,0.5);
	\path [draw, ->, thick] (ew1) -- ++(-0.4,0.5);
	\path [draw, <-, thick] (ew2) -- ++(0.4,0.5);

	\draw (ey1) to[bend left,looseness=2,in=70,out=110] (ez1);
	\draw (ey2) to[bend right,looseness=2,in=-70,out=-110] (ez2);
	\draw (ew1) to[bend right,looseness=1.2,in=-70,out=-110] (ew2);
\end{tikzpicture}
 \caption{A smooth orientation of the circuits in $F-\{x_1y_1, y_2x_2\}\cup M$ and of the circuits in $C$ such that each consistent on the edges of $N_i$ at distance $1$ from $x_i$, for $i\in\{1,2\}$, and on the edges of $W$ at distance $1$ from $y_i$ but not incident with $y_{3-i}$, for $i\in\{1,2\}$.}
\label{fig:2foeo}
\end{figure}
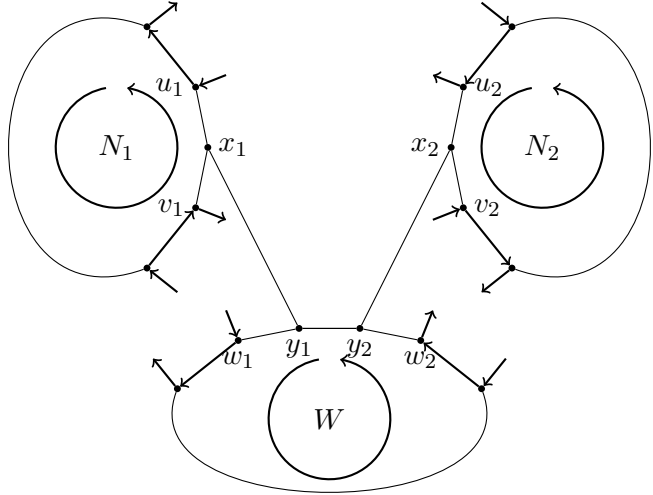

\section{Algorithm}\label{sec:algorithm}
We propose two algorithms for computationally verifying whether or not a given $3$-edge-connected cubic graph has Frank number $2$, i.e.\ a heuristic and an exact algorithm. Note that the Frank number for $3$-edge-connected cubic graphs is always at least $2$. 
Our algorithms are intended for graphs which are not $3$-edge-colourable, since $3$-edge-connected $3$-edge-colourable graphs have Frank number $2$ (cf.\ Theorem~\ref{thm:4flow}).

The first algorithm is a heuristic algorithm, which makes use of Theorem~\ref{thm:2OddCycles} and Theorem~\ref{thm:2Odd1Even}. Hence, it can only be used for cyclically $4$-edge-connected cubic graphs. For every $2$-factor in the input graph $G$, we verify if one of the configurations of these theorems is present. If that is the case, the graph has Frank number $2$. 

More specifically, we look at every $2$-factor of $G$ by generating every perfect matching and looking at its complement. We then count how many odd circuits there are in the $2$-factor under investigation. If there are precisely two odd circuits, then we check for every edge connecting the two odd circuits whether or not the conditions of Theorem~\ref{thm:2OddCycles} hold. If they hold for one of these edges, we stop the algorithm and return that the graph has Frank number $2$. If these conditions do not hold for any of these edges or if there are none, we check for all triples of edges $x_1y_1, y_1y_2, y_2x_2$, where $x_1$ and $x_2$ lie on different odd circuits and $y_1$ and $y_2$ lie on the same even circuit of our $2$-factor, whether the conditions of Theorem~\ref{thm:2Odd1Even} hold. If they do, then $G$ has Frank number $2$ and we stop the algorithm. The pseudocode of this algorithm can be found in Algorithm~\ref{alg:heuristic}.

Note that in practice, when checking the conditions of Theorem~\ref{thm:2OddCycles} and Theorem~\ref{thm:2Odd1Even}, we only consider one maximal matching as defined in the statements of the theorems. 
This has no effect on the correctness of the heuristic algorithm, but if we choose a matching for which the conditions of Theorem~\ref{thm:2OddCycles} and Theorem~\ref{thm:2Odd1Even} do not hold, the heuristic will not be sufficient to decide whether or not the Frank number is $2$. The second algorithm is then needed to decide this. As we will see in the results section, this approach is sufficient to generate all graphs in the relevant class up to all orders for which it is feasible to generate them.

The second algorithm is an exact algorithm for determining whether or not a $3$-edge-connected cubic graph has Frank number $2$. The pseudocode of this algorithm can be found in Algorithm~\ref{alg:fn=2}. For a graph $G$, we start by considering each of its strong orientations $(G,o)$ and try to find a complementary orientation $(G,o')$ such that every edge is deletable in either $(G,o)$ or $(G,o')$. First, we check if there is a vertex in $G$ for which none of its adjacent edges are deletable in $(G,o)$. If this is the case, there exists no complementary orientation as no orientation of a cubic graph can have three deletable edges incident to the same vertex.
If $(G,o)$ does not contain such a vertex, we look for a complementary orientation using some tricks to reduce the search space. 

More precisely, we first we start with an empty \textit{partial orientation}, i.e.\ a directed spanning subgraph of some orientation of $G$, and fix the orientation of some edge. Note that we do not need to consider the opposite orientation of this edge, since an orientation of a graph in which all arcs are reversed has the same set of deletable edges as the original orientation.

We then recursively orient edges of $G$ that have not yet been oriented. After orienting an edge, the rules of Lemma~\ref{lemma:rules}
may enforce the orientation of edges which are not yet oriented. We orient them in this way before proceeding with the next edge. This heavily restricts the number edges which need to be added. As soon as a complementary orientation is found, we can stop the algorithm and return that the graph $G$ has Frank number $2$. If for all strong orientations of $G$ no such complementary orientation is found, then the Frank number of $G$ is higher than~$2$.  

Since the heuristic algorithm is much faster than the exact algorithm, we will first apply the heuristic algorithm. After this we will apply the exact algorithm for those graphs for which the heuristic algorithm was unable to decide whether or not the Frank number is $2$. In Section~\ref{sec:results}, we give more details on how many graphs pass this heuristic algorithm.

An implementation of these algorithms can be found on GitHub~\cite{GMR23Program}. Our implementation uses bitvectors to store adjacency lists and lists of edges and uses bitoperations to efficiently manipulate these lists.

\makeatletter
\let\OldStatex\Statex
\renewcommand{\Statex}[1][3]{%
  \setlength\@tempdima{\algorithmicindent}%
  \OldStatex\hskip\dimexpr#1\@tempdima\relax}
\makeatother

\begin{algorithm}[!htb]
\caption{$\operatorname{heuristicForFrankNumber2}(\text{Graph }G)$}\label{alg:heuristic}
\begin{algorithmic}[1]
\For{each perfect matching $F$}
	\State Store odd circuits of $C:= G - F$ in $\mathcal{O} = \{N_1, \ldots, N_k\}$
	\If{$\lvert \mathcal{O} \rvert$ is not $2$}
		\State Continue with the next perfect matching
	\EndIf
	\ForAll{edges $x_1x_2$ with $x_1\in V(N_1), x_2\in V(N_2)$}
            \State // Test if Theorem~\ref{thm:2OddCycles} can be applied
			\State Store a maximal matching of $C - \{x_1,x_2\}$ in $M$ 
			\State Denote the neighbours of $x_1$ and $x_2$ in $C$ by $u_1, v_1$ and $u_2, v_2$, respectively
            \State Denote the set of edges of $N_1$ and $N_2$ at distance $1$ from $x_1$ and $x_2$ by $Z$
			\State Create an empty partial orientation $(F - \{x_1,x_2\}\cup M,o)$
			\ForAll{$x\in \{u_1, v_1, u_2, v_2\}$}
				\If{the circuit in $F - \{x_1,x_2\}\cup M$ containing $x$ is not yet oriented}
					\State Orient the circuit in $F - \{x_1,x_2\}\cup M$ containing $x$ such that it is smooth
                    \Statex[4] and maintaining consistency on $Z$ with some orientation of $C$ if possible.
				\EndIf
			\EndFor
			\If{$(F - \{x_1,x_2\}\cup M,o)$ can be extended to a smooth orientation consistent on 
            \Statex[2] $Z$ with some smooth orientation of $C$}
				\State \Return True \label{line:first} // Theorem~\ref{thm:2OddCycles} applies
			\EndIf
    \EndFor
    \ForAll{pairs of edges $x_1y_1, x_2y_2$ with $x_1\in V(N_1), x_2 \in V(N_2)$ and $y_1, y_2$ adjacent
    \Statex[1] and on the same even circuit $W$ of $C$}
                    \State // Test if Theorem~\ref{thm:2Odd1Even} can be applied
					\State Store a maximal matching of $C - \{x_1, y_1, y_2, x_2\}$ in $M$
					\State Denote the neighbours of $x_1$ and $x_2$ in $C$ by $u_1, v_1$ and $u_2, v_2$, respectively
					\State Denote the neighbour of $y_1$ in $C - y_2$ by $w_1$
                    and of $y_2$ in $C - y_1$ by $w_2$
                    \State Denote the set of edges of $N_i$ at distance $1$ from $x_i$ and of $W$ at distance $1$ from
                    \Statex[2] $y_i$ but not incident with $y_{3-i}$ by $Z$
					\State Create an empty partial orientation $(F - \{x_1, y_1, y_2, x_2\}\cup M,o)$
					\ForAll{$x\in \{u_1, v_1, u_2, v_2, w_1, w_2\}$}
						\If{the circuit in $F - \{x_1, y_1, y_2, x_2\}\cup M$ with $x$ is not oriented in $(G,o)$}
							\State Orient the circuit in $F - \{x_1, y_1, y_2, x_2\}\cup M$ containing $x$ such that it is \Statex[4] smooth and maintaining consistency on $Z$ with some orientation of $C$
						\EndIf
					\EndFor
					\If{$(F - \{x_1, y_1, y_2, x_2\}\cup M,o)$ can be extended to a smooth orientation consistent 
                    \Statex[2] on $Z$ with some orientation of $C$}
                        \If{$G\sim x_1y_1\sim x_2y_2$ is not cyclically $3$-edge-connected}
                            \State Continue with for loop
                        \EndIf
                        \State // Check cycle-separating edge-set condition
						\ForAll {pairs of edges $e_1, e_2$ in $F - \{x_1, y_1, y_2, x_2\}\cup M$}
							\ForAll {$e\in \{u_1x_1, w_1y_1, u_2x_2\}$}
								\If{$\{e, e_1, e_2\}$ cyclically separates $G - x_1y_1 - x_2y_2$}
									\State \Return True\label{line:second} // Theorem~\ref{thm:2Odd1Even} applies
								\EndIf
							\EndFor
						\EndFor
					\EndIf
	\EndFor
\EndFor
\State \Return False
\end{algorithmic}
\end{algorithm}

\begin{theorem}\label{thm:correctness1}
	Let $G$ be a cyclically $4$-edge-connected cubic graph. If Algorithm~\ref{alg:heuristic} is applied to $G$ and returns True, $G$ has Frank number $2$.
\end{theorem}
\begin{proof}
	Suppose the algorithm returns True for $G$. This happens in a specific iteration of the outer for-loop corresponding to a perfect matching $F$. The complement of $F$ is a $2$-factor, say $C$, and since the algorithm returns True, $C$ has precisely two odd circuits, say $N_1$ and $N_2$, and possibly some even circuits.
	
Suppose first that the algorithm returns True on Line~\ref{line:first}. Then there is an edge $x_1x_2$ in $G$ with $x_1 \in V(N_1)$ and $x_2\in V(N_2)$, a maximal matching $M$ of $C - \{x_1, x_2\}$ and orientations ($F - \{x_1x_2\}\cup M, o_1)$ and $(N_1\cup N_2, o_2)$ which are consistent on the edges of $N_i$ at distance 1 from $x_i$. Now by Theorem~\ref{thm:2OddCycles} it follows that $G$ has Frank number $2$.
	
Now suppose that the algorithm returns True on Line~\ref{line:second}. Then there are edges $x_1y_1, y_1y_2$ and $y_2x_2$ such that $x_1\in V(N_1)$, $x_2\in V(N_2)$ and $y_1,y_2\in V(W)$ where $W$ is some even circuit in $C$. Since the algorithm returns True, there is a maximal matching $M$ of $C - \{x_1,y_1,y_2,x_2\}$ and smooth orientations $(F - \{x_1x_2\}\cup M,o_1)$ and $(N_1\cup N_2\cup W, o_2)$ which are consistent on the edges of $N_i$ at distance 1 from $x_i$ and on the edges of $W$ at distance 1 from $y_i$ and are not incident with $y_{3-i}$ for $i\in\{1,2\}$.

Denote the neighbours of $x_1$ and $x_2$ in $C$ by $u_1, v_1$ and $u_2, v_2$, respectively and denote the neighbour of $y_1$ in $C - y_2$ by $w_1$ and the neighbour of $y_2$ in $C - y_1$ by $w_2$. Since no triple $e, e_1, e_2$, where $e\in \{u_1x_1, w_1y_1, u_2x_2\}$, $e_1, e_2\in E(F - \{x_1y_1, x_2y_2\}\cup M)$, is a cycle-separating edge-set of $G - \{x_1y_1,x_2y_2\}$, $G \sim x_1y_1\sim x_2y_2$ has no cycle-separating edge-set $\{e, e_1, e_2\}$, where $e\in \{u_1v_1,u_2v_2, w_1w_2\}$ and $e_1, e_2\in E(F - \{x_1y_1, x_2y_2\}\cup M)$. Now by Theorem~\ref{thm:2Odd1Even} it follows that $G$ has Frank number 2.
\end{proof}

\begin{algorithm}[!htb]
	\caption{frankNumberIs2(Graph $G$)}\label{alg:fn=2}
	\begin{algorithmic}[1]
		\ForAll{orientations $(G,o)$ of $G$} \label{line:orientationsLoop}
			\If{$(G,o)$ is not strong}
				\State Continue with next orientation
			\EndIf
			\State Store deletable edges of $(G,o)$ in a set $D$
            \ForAll{$v\in V(G)$}
			    \If{no edge incident to $v$ is deletable}
				    \State Continue with next orientation
			    \EndIf
            \EndFor
			\State Create empty partial orientation $(G,o')$ of $G$
			\State Choose an edge $xy$ in $G$ and fix orientation $x\rightarrow y$ in $o'$\label{line:choiceOfArc} 
			\If{not canAddArcsRecursively($(G, o')$, $D$, $x\rightarrow y$)} // Algorithm~\ref{alg:recursion}
				\State Continue loop with next orientation
			\EndIf
            \If{canCompleteOrientation($(G, o')$, $D$)} // Algorithm~\ref{alg:canCompleteOrientation}
                \State \Return True
            \EndIf
		\EndFor
		\State \Return False
	\end{algorithmic}
\end{algorithm}

\begin{algorithm}[!htb]
    \caption{canCompleteOrientation(Partial Orientation $(G, o')$, Set $D$}\label{alg:canCompleteOrientation}
    \begin{algorithmic}[1]
        \If{all edges are oriented in $(G,o')$}
            \If{$D\cup D(G,o') = E(G)$}\label{line:fnCheck}
                \State \Return True
            \EndIf
            \State \Return False
        \EndIf
        
        \State // $(G,o')$ still has unoriented edges
        \State Store a copy of $(G,o')$ in $(G,o'')$
        \State Choose an edge $uv$ unoriented in $(G,o')$
        \If{canAddArcsRecursively($(G,o')$, $D$, $u\rightarrow v$)}
            \If{canCompleteOrientation($(G,o')$, $D$)} 
                \State \Return True
            \EndIf
        \EndIf
        \State Reset $o'$ using $o''$
        \If{canAddArcsRecursively($(G,o')$, $D$, $v\rightarrow u$)}
            \If{canCompleteOrientation($(G,o')$, $D$)}
                \State \Return True
            \EndIf
        \EndIf
        \State \Return False
    \end{algorithmic}
\end{algorithm}

We will use the following Lemma for the proof of the exact algorithm's correctness.

\begin{lemma}\label{lemma:rules}
	Let $G$ be a cubic graph with $fn(G) = 2$ and let $(G,o)$ and $(G,o')$ be two orientations of $G$ such that every edge $e\in E(G)$ is deletable in either $(G,o)$ or $(G,o')$. Then the following hold for $(G,o')$:
	\begin{enumerate}
		\item every vertex has at least one incoming and one outgoing edge in $(G,o')$,
		\item let $uv\not\in D(G, o)$, then $u$ has one incoming and one outgoing edge in $(G-uv, o')$,
		\item let $uv, vw\not\in D(G, o)$, then they are oriented either $u\rightarrow v$, $w\rightarrow v$ or $v\rightarrow u$, $v\rightarrow w$ in $(G,o')$.
	\end{enumerate}
\end{lemma}
\begin{proof}
We now prove each of the three properties:
	\begin{enumerate}
		\item Let $u$ be a vertex such that all its incident edges are either outgoing or incoming in $(G,o')$. Clearly none of these edges can be deletable in $(G,o')$. Since there is some edge $ux$ not in $D(G, o)$. We get a contradiction.
		\item Let $uv\not \in D(G, o)$ and let the remaining edges incident to $u$ be either both outgoing or both incoming in $(G,o')$. Then $uv$ is not deletable in $(G,o')$ since all oriented paths to (respectively, from) $u$ pass through $uv$.
		\item Suppose without loss of generality that we have $u\rightarrow v$ and $v\rightarrow w$ in $(G, o')$. If the remaining edge incident to $v$ is outgoing, then $uv$ is not deletable in $(G, o')$. If the remaining edge is incoming, then $vw$ is not deletable in $(G,o')$.
	\end{enumerate}
\end{proof}
\begin{theorem}\label{thm:correctness2}
	Let $G$ be a cubic graph. Algorithm~\ref{alg:fn=2} applied to $G$ returns True if and only if $G$ has Frank number $2$.
\end{theorem}
\begin{proof}
Suppose that $\operatorname{frankNumberIs2(G)}$ returns True. Then there exist two orientations $(G,o)$ and $(G,o')$ for which $D(G, o)\cup D(G, o') = E(G)$. Hence, $fn(G)=2$.

Conversely, let $fn(G) = 2$.
We will show that Algorithm~\ref{alg:fn=2} returns True. Let $(G, o_1)$ and $(G, o_2)$ be orientations of $G$ such that every edge of $G$ is deletable in either $(G, o_1)$ or $(G, o_2)$. Every iteration of the loop of Line~\ref{line:orientationsLoop}, we consider an orientation of $G$. If the algorithm returns True before we consider $(G, o_1)$ in this loop, the proof done. So without loss of generality, suppose we are in the iteration where $(G, o_1)$ is the orientation under consideration in the loop of Line~\ref{line:orientationsLoop}.

Without loss of generality assume that the orientation of $xy$ we fix on Line~\ref{line:choiceOfArc} is in $(G, o_2)$. (If not, reverse all edges of $(G, o_2)$ to get an orientation with the same set of deletable edges.) Let $(G, o')$ be a partial orientation of $G$ and assume that all oriented edges correspond to $(G, o_2)$. Let $u\rightarrow v$ be an arc in $(G, o_2)$.  If $u\rightarrow v$ is present in $(G, o')$, then canAddArcsRecursively($G$, $D(G, o)$, $o'$, $u\rightarrow v$) (Algorithm~\ref{alg:recursion}) returns True and no extra edges become oriented in $(G, o')$. If $u\rightarrow v$ is not present in $(G, o')$, it gets added on Line~\ref{line:addArc} of Algorithm~\ref{alg:recursion}, since the if-statement on Line~\ref{line:violateRules} of Algorithm~\ref{alg:recursion} will return True by Lemma~\ref{lemma:rules}. Note that this is the only place where an arc is added to $(G, o')$ in Algorithm~\ref{alg:recursion}. Hence, if we only call Algorithm~\ref{alg:recursion} on arcs present in $(G, o_2)$, then all oriented edges $e$ of $(G, o')$ will always be oriented as $o_2(e)$. Now we will show that Algorithm~\ref{alg:recursion} indeed only calls itself on arcs in $(G, o_2)$.
	
Again, suppose $u\rightarrow v$ is an arc in $(G, o_2)$, that it is not yet oriented in $(G, o')$ and that every oriented edge $e$ of $(G, o')$ has orientation $o_2(e)$. The call canAddArcsRecursively($G$, $D(G, o)$, $o'$, $u\rightarrow v$) can only call itself on Line~\ref{line:call_self}, i.e.\ in Algorithm~\ref{alg:orientFixedEdges}. We show that in all cases after orienting $uv$ as $u\rightarrow v$ in $(G,o')$, the call to Algorithm~\ref{alg:recursion} only happens on arcs oriented as in $(G,o_2)$.

Suppose $u$ has two outgoing and no incoming arcs in $(G, o')$. Let $ux$ be the final unoriented edge incident to $u$. Then $(G, o_2)$ must have arc $x\rightarrow u$, otherwise it has three outgoing arcs from the same vertex. Now suppose $v$ has two incoming and no outgoing arcs in $(G, o')$. Let $vx$ be the final unoriented edge incident to $v$. Then $(G, o_2)$ must have arc $v\rightarrow x$, otherwise it has three incoming arcs to the same vertex.

Suppose $uv$ is deletable in $(G, o_1)$. Let $ux$ also be deletable in $(G, o_1)$. Denote the final edge incident to $u$ by $uy$. Clearly, $uy$ cannot be deletable in $(G, o_1)$. Hence, it is deletable in $(G, o_2)$. If $(G, o_2)$ contains $u\rightarrow x$, then $uy$ is not deletable in $o_2$. Hence, $(G, o_2)$ contains $x\rightarrow u$. Let $vx$ be a deletable edge of $(G, o_1)$ and denote the final edge incident to $v$ by $vy$. Since $vy$ cannot be deletable in $(G, o_1)$, $(G, o_2)$ must contain arc $v\rightarrow x$. Suppose that the edges incident with $u$ which are not $uv$ are both not in $D(G, o_1)$. Then they must be oriented incoming to $u$ in $(G, o_2)$. Similarly, if the edges incident with $v$ which are not $uv$ are both not in $D(G, o_1)$, they must both be outgoing from $v$ in $(G, o_2)$. 

Finally, suppose that $uv$ is not a deletable edge in $(G, o_1)$. Suppose that $(G, o')$ still has one unoriented edge incident to $u$, say $ux$. If the other incident edges are one incoming and one outgoing from $u$, then $(G, o_2)$ contains the arc $u\rightarrow x$. Otherwise, $uv$ cannot be deletable in $(G, o_2)$. Similarly, if $(G, o')$ still has one unoriented edge incident to $v$, say $vx$ and the remaining incident edges are one incoming and one outgoing, then the arc $x\rightarrow v$ must be present in $(G, o_2)$. Otherwise, $uv$ cannot be deletable in $(G, o_2)$. If $ux$ is not deletable in $(G, o_1)$ $x\neq v$. Then $(G, o_2)$ contains the arc $u\rightarrow x$. Otherwise, not both of $uv$ and $ux$ can be deletable in $(G, o_2)$. Similarly, if $vy$ is not deletable in $(G, o_1)$ and $y\neq u$, then $(G, o_2)$ must contain the arc $y\rightarrow v$. Otherwise, not both of $uv$ and $vy$ can be deletable in $(G, o_2)$.

This shows that all calls of canAddArcsRecursively($G$, $D(G, o)$, $o'$, $u\rightarrow v$) to itself, where all oriented edges of $(G, o')$ and $u->v$ are oriented as in $(G, o_2)$,  have as the fourth parameter an arc oriented as in $(G, o_2)$. Since, in Algorithm~\ref{alg:canCompleteOrientation}, we keep orienting edges until $(G,o')$ is completely oriented and try to orient edge $uv$ as $v\rightarrow u$ if $(G, o')$ cannot be completed with $u\rightarrow v$, it follows by induction that unless Algorithm~\ref{alg:fn=2} returns True in some other case, it will return True when $(G, o') = (G, o_2)$.
\end{proof}

\begin{algorithm}[!htb]
	\caption{canAddArcsRecursively(Partial Orientation $(G, o')$, Set $D$, Arc $u\rightarrow v$)}\label{alg:recursion}
	\begin{algorithmic}[1]
        \State // Check if $u\rightarrow v$ can be added and recursively orient edges for which the orientation is enforced by the rules of Lemma~\ref{lemma:rules}
		\If{$u\rightarrow v$ is present in $(G,o')$}
			\State \Return True
		\EndIf
  		\If{$v\rightarrow u$ is present in $(G,o')$}
			\State \Return False
		\EndIf
        \If{adding $u\rightarrow v$ violates rules of Lemma~\ref{lemma:rules}} // Algorithm~\ref{alg:canAdd} in Appendix~\ref{app:algorithms}\label{line:violateRules}
            \State \Return False 
        \EndIf
		\State Add $u\rightarrow v$ to $(G,o')$\label{line:addArc}
		\If{the orientation of edges enforced by Lemma~\ref{lemma:rules} yields a contradiction} // Algorithm~\ref{alg:orientFixedEdges} in Appendix~\ref{app:algorithms}\label{line:call_self}
            \State \Return False
        \EndIf
		\State \Return True
	\end{algorithmic}
\end{algorithm}

\subsection{Results}\label{sec:results}
Since by Theorem~\ref{thm:4flow} all $3$-edge-connected $3$-edge-colourable (cubic) graphs have Frank number $2$, in this section we will focus on \textit{non}-$3$-edge-colourable cubic graphs, i.e.\ \textit{snarks}.

In \cite{BGHM13} Brinkmann et al.\ determined all cyclically $4$-edge-connected snarks up to order $34$ and those of girth at least $5$ up to order $36$. 
This was later extended with all cyclically $4$-edge-connected snarks on $36$ vertices as well \cite{GMS19}. These lists of snarks can be obtained from the House of Graphs \cite{CDG23} at: \url{https://houseofgraphs.org/meta-directory/snarks}. Using our implementation of Algorithms~\ref{alg:heuristic} and~\ref{alg:fn=2}, we tested for all cyclically $4$-edge-connected snarks up to $36$ vertices if they have Frank number $2$ or not. This led to the following result.
\begin{proposition}\label{prop:petersenComp}
	The Petersen graph is the only cyclically $4$-edge-connected snark up to order $36$ which has Frank number not equal to $2$. 
\end{proposition}
This was done by first running our heuristic Algorithm~\ref{alg:heuristic} on these graphs. It turns out that there are few snarks in which neither the configuration of Theorem~\ref{thm:2OddCycles} nor the configuration of Theorem~\ref{thm:2Odd1Even} are present. For example: for more than $99.97\%$ of the cyclically $4$-edge-connected snarks of order $36$, Algorithm~\ref{alg:heuristic} is sufficient to determine that their Frank number is $2$ (see Table~\ref{tab:snarksPassingHeuristic} in Appendix~\ref{app:snarksPassingHeuristic} for more details). Thus we only had to run our exact Algorithm~\ref{alg:fn=2} (which is significantly slower than the heuristic) on the graphs for which our heuristic algorithm failed. In total about 214 CPU days of computation time was required to prove Proposition~\ref{prop:petersenComp} using Algorithm~\ref{alg:heuristic} and~\ref{alg:fn=2} (see Table~\ref{tab:runtimes} in Appendix~\ref{app:snarksPassingHeuristic} for more details).

In \cite{Ja85} Jaeger defines a snark $G$ to be a \textit{strong snark} if for every edge $e\in E(G)$, $G\sim e$, i.e.\ the unique cubic graph such that $G - e$ is a subdivision of $G\sim e$, is not $3$-edge-colourable. Hence, a strong snark containing a $2$-factor which has precisely two odd circuits, has no edge $e$ connecting those two odd circuits, i.e.\ the configuration of Theorem~\ref{thm:2OddCycles} cannot be present. Therefore, they might be good candidates for having Frank number greater than $2$.

In~\cite{BGHM13} it was determined that there are $7$ strong snarks on $34$ vertices having girth at least $5$, $25$ strong snarks on $36$ vertices having girth at least $5$ and no strong snarks of girth at least $5$ of smaller order. By Proposition~\ref{prop:petersenComp}, their Frank number is $2$. In~\cite{BG17} it was determined that there are at least $298$ strong snarks on $38$ vertices having girth at least $5$ and the authors of~\cite{BG17} speculate that this is the complete set. We found the following.
\begin{observation}
    The $298$ strong snarks of order $38$ determined in~\cite{BG17} have Frank number~$2$.
\end{observation}
These snarks can be obtained from the House of Graphs~\cite{CDG23} by searching for the keywords ``strong snark". 

The configurations of Theorem~\ref{thm:2OddCycles} and Theorem~\ref{thm:2Odd1Even} also cannot occur in snarks of \emph{oddness}~$4$, i.e.\ the smallest number of odd circuits in a $2$-factor of the graph is $4$. Hence, these may also seem to be good candidates for having Frank number greater than $2$. In \cite{GMS19,GMS20} it was determined that the smallest snarks of girth at least $5$ with oddness $4$ and cyclic edge-connectivity $4$ have order $44$ and that there are precisely $31$ such graphs of this order. We tested each of these and found the following.
\begin{observation}
    Let $G$ be a snark of girth at least $5$, oddness $4$, cyclic edge-connectivity $4$ and order $44$. Then $fn(G) = 2$.
\end{observation}
These snarks of oddness $4$ can be obtained from the House of Graphs~\cite{CDG23} at \url{https://houseofgraphs.org/meta-directory/snarks}.

\subsection{Correctness Testing}

The correctness of our algorithm was shown in Theorem~\ref{thm:correctness1} and Theorem~\ref{thm:correctness2}. We also performed several tests to verify that our implementations are correct.

H\"orsch and Szigeti proved in \cite{HS21} that the Petersen graph has Frank number $3$. In~\cite{BB24_quest} Bar{\'a}t and Bl{\'a}zsik showed that both Blanu\v{s}a snarks and every flower snark has Frank number $2$. We verified that for the Petersen graph both Algorithm~\ref{alg:heuristic} and Algorithm~\ref{alg:fn=2} give a negative result, confirming that its Frank number is larger than $2$. For the Blanu\v{s}a snarks and the flower snarks up to $40$ vertices Algorithm~\ref{alg:fn=2} always shows the graph has Frank number $2$ and the heuristic Algorithm~\ref{alg:heuristic} is able to show this for a subset of these graphs.

We also ran our implementation of Algorithm~\ref{alg:fn=2} on the cyclically $4$-edge-connected snarks up to $30$ vertices without running Algorithm~\ref{alg:heuristic} first. Results were in complete agreement with our earlier computation, i.e.\ Algorithm~\ref{alg:fn=2} independently confirmed that each of the snarks which have Frank number $2$ according to Algorithm~\ref{alg:heuristic} indeed have Frank number $2$. For the cyclically $4$-edge-connected snarks on $30$ vertices Algorithm~\ref{alg:fn=2} took approximately $46$ hours. Our heuristic Algorithm~\ref{alg:heuristic} found the same results for all but $307$ graphs in approximately $42$ seconds.

During the computation of Algorithm~\ref{alg:heuristic} on the strong snarks, we verified it never returned True for the configuration of Theorem~\ref{thm:2OddCycles} and that it returned False for all snarks of oddness $4$ mentioned earlier. 

We also implemented a method for finding the actual orientations after Algorithm~\ref{alg:heuristic} detects one of the configurations. We checked for all cyclically $4$-edge-connected snarks up to $32$ vertices in which one of these configurations is found, whether the deletable edges for these two orientations form the whole edge set. This was always the case. 

Another test was performed using a brute force algorithm which generates all strongly connected orientations of the graph and checks for every pair of these orientations whether the union of the deletable edges of this pair of orientations is the set of all edges. We were able to do this for all cyclically-$4$-edge-connected snarks up to $26$ vertices and obtained the same results as with our other method. Note that this method is a lot slower than Algorithm~\ref{alg:heuristic} and~\ref{alg:fn=2}. For order $26$ this took approximately 152 hours, while using Algorithm~\ref{alg:heuristic} and~\ref{alg:fn=2} this took approximately $1$ second.

Our implementation of Algorithm~\ref{alg:heuristic} and Algorithm~\ref{alg:fn=2} is open source and can be found on GitHub~\cite{GMR23Program} where it can be inspected and used by others.

\subsubsection*{Acknowledgements}
We would like to thank J{\'a}nos Bar{\'a}t for useful discussions. 
The research of Edita Máčajová was partially supported by the grant No.\ APVV-19-0308 of the Slovak Research and Development Agency and the grant VEGA 1/0743/21.
The research of Jan Goedgebeur and Jarne renders was supported by Internal Funds of KU Leuven. 
Several of the computations for this work were carried out using the supercomputer infrastructure provided by the VSC (Flemish Supercomputer Center), funded by the Research Foundation Flanders (FWO) and the Flemish Government.

\clearpage
%
%
%
\bibliographystyle{plainurl}
\bibliography{ref}

\begin{thebibliography}{10}

\bibitem{BB24_quest}
J.~Bar{\'a}t and Z.L. Bl{\'a}zsik.
\newblock Quest for graphs of {Frank} number $3$.
\newblock {\em Australas. J. Combin.}, 88:52--76, 2024.

\bibitem{BB24}
J.~Barát and Z.L. Blázsik.
\newblock Improved upper bound on the frank number of 3-edge-connected graphs.
\newblock {\em Eur. J. Comb.}, 118:103913, 2024.
\newblock \href {https://doi.org/https://doi.org/10.1016/j.ejc.2023.103913}
  {\path{doi:https://doi.org/10.1016/j.ejc.2023.103913}}.

\bibitem{BG17}
G.~Brinkmann and J.~Goedgebeur.
\newblock Generation of {Cubic} {Graphs} and {Snarks} with {Large} {Girth}.
\newblock {\em J. Graph Theory}, 86(2):255--272, 2017.
\newblock \href {https://doi.org/10.1002/jgt.22125}
  {\path{doi:10.1002/jgt.22125}}.

\bibitem{BGHM13}
G.~Brinkmann, J.~Goedgebeur, J.~Hägglund, and K.~Markström.
\newblock Generation and properties of snarks.
\newblock {\em J. Comb. Theory. Ser. B}, 103(4):468--488, 2013.
\newblock \href {https://doi.org/10.1016/j.jctb.2013.05.001}
  {\path{doi:10.1016/j.jctb.2013.05.001}}.

\bibitem{CDG23}
K.~Coolsaet, S.~D’hondt, and J.~Goedgebeur.
\newblock House of {Graphs} 2.0: {A} database of interesting graphs and more.
\newblock {\em Discret. Appl. Math.}, 325:97--107, 2023.
\newblock \href {https://doi.org/10.1016/j.dam.2022.10.013}
  {\path{doi:10.1016/j.dam.2022.10.013}}.

\bibitem{GMR23Program}
J.~Goedgebeur, E.~M\'a\v{c}ajov\'a, and J.~Renders.
\newblock {Frank-Number}, 1 2023.
\newblock URL: \url{https://github.com/JarneRenders/Frank-Number}.

\bibitem{WG2023}
J.~Goedgebeur, E.~Máčajová, and J.~Renders.
\newblock On the frank number and nowhere-zero flows on graphs.
\newblock In {\em Proceedings of the 49th International Workshop on
  Graph-Theoretic Concepts in Computer Science (WG2023), Fribourg, Switzerland,
  LNCS 14093}, pages 363--375. Springer, 2023.
\newblock \href {https://doi.org/10.1007/978-3-031-43380-1_26}
  {\path{doi:10.1007/978-3-031-43380-1_26}}.

\bibitem{GMS19}
J.~Goedgebeur, E.~Máčajová, and M.~Škoviera.
\newblock Smallest snarks with oddness 4 and cyclic connectivity 4 have order
  44.
\newblock {\em Ars Math. Contemp.}, 16(2):277--298, 2019.
\newblock \href {https://doi.org/10.26493/1855-3974.1601.e75}
  {\path{doi:10.26493/1855-3974.1601.e75}}.

\bibitem{GMS20}
J.~Goedgebeur, E.~Máčajová, and M.~Škoviera.
\newblock The smallest nontrivial snarks of oddness 4.
\newblock {\em Discret. Appl. Math.}, 277:139--162, 2020.
\newblock \href {https://doi.org/10.1016/j.dam.2019.09.020}
  {\path{doi:10.1016/j.dam.2019.09.020}}.

\bibitem{HS21}
F.~Hörsch and Z.~Szigeti.
\newblock Connectivity of orientations of 3-edge-connected graphs.
\newblock {\em Eur. J. Comb.}, 94:103292, 2021.
\newblock \href {https://doi.org/10.1016/j.ejc.2020.103292}
  {\path{doi:10.1016/j.ejc.2020.103292}}.

\bibitem{Ja85}
F.~Jaeger.
\newblock A {Survey} of the {Cycle} {Double} {Cover} {Conjecture}.
\newblock In B.~R. Alspach and C.~D. Godsil, editors, {\em North-{Holland}
  {Mathematics} {Studies}}, volume 115 of {\em Annals of {Discrete}
  {Mathematics} (27): {Cycles} in {Graphs}}, pages 1--12. North-Holland, 1985.
\newblock \href {https://doi.org/10.1016/S0304-0208(08)72993-1}
  {\path{doi:10.1016/S0304-0208(08)72993-1}}.

\bibitem{Ko16}
D.~Kőnig.
\newblock Gráfok és alkalmazásuk a determinánsok {Žs} a halmazok
  elméletére.
\newblock {\em Matematikai és Természettudományi Értesítő}, 34:104--119,
  1916.

\bibitem{Na60}
C.~St J.~A. Nash-Williams.
\newblock On {Orientations}, {Connectivity} and {Odd}-{Vertex}-{Pairings} in
  {Finite} {Graphs}.
\newblock {\em Can. J. Math.}, 12:555--567, 1960.
\newblock Publisher: Cambridge University Press.
\newblock \href {https://doi.org/10.4153/CJM-1960-049-6}
  {\path{doi:10.4153/CJM-1960-049-6}}.

\bibitem{Se79}
P.~D. Seymour.
\newblock On {Multi}-{Colourings} of {Cubic} {Graphs}, and {Conjectures} of
  {Fulkerson} and {Tutte}.
\newblock {\em Proc. London Math. Soc.}, s3-38(3):423--460, 1979.
\newblock \href {https://doi.org/10.1112/plms/s3-38.3.423}
  {\path{doi:10.1112/plms/s3-38.3.423}}.

\bibitem{Se81}
P.~D Seymour.
\newblock Nowhere-zero 6-flows.
\newblock {\em J. Comb. Theory. Ser. B}, 30(2):130--135, 1981.
\newblock \href {https://doi.org/10.1016/0095-8956(81)90058-7}
  {\path{doi:10.1016/0095-8956(81)90058-7}}.

\bibitem{Tu54}
W.~T. Tutte.
\newblock A {Contribution} to the {Theory} of {Chromatic} {Polynomials}.
\newblock {\em Can. J. Math.}, 6:80--91, 1954.
\newblock \href {https://doi.org/10.4153/CJM-1954-010-9}
  {\path{doi:10.4153/CJM-1954-010-9}}.

\end{thebibliography}
%




\newpage

\appendix
\section{Appendix}
\subsection{Algorithms}\label{app:algorithms}
\begin{algorithm}[!htb]
	\caption{canAddArc(Partial Orientation $(G,o')$, Set $D$, Arc $u\rightarrow v$)}\label{alg:canAdd}
	\begin{algorithmic}[1]
        \State // Check if the addition of $u\rightarrow v$ will not violate rules of Lemma~\ref{lemma:rules}
		\If{$u$ has two outgoing arcs or $v$ has two incoming arcs in $(G,o')$}
			\State \Return False
		\EndIf
		\If{$uv\in D$}
			\ForAll{edges $e$ incident to $u$ in $G$}
				\If{$e\in D$ and $e$ is an outgoing arc of $u$ in $(G,o')$}
					\State \Return False
				\EndIf
			\EndFor
			\ForAll{edges $e$ incident to $v$ in $G$}
				\If{$e\in D$ and $e$ is an incoming arc of $v$ in $(G,o')$}
					\State \Return False
				\EndIf
			\EndFor
		\Else
			\If{in $(G,o')$, $u$ has two incoming arcs or
            \Statex[1] $v$ has two outgoing arcs or
            \Statex[1] $u$ has an incoming arc whose corresponding edge is not in $D$ or 
			\Statex[1] $v$ has an outgoing arc whose corresponding edge is not in $D$}
				\State \Return False
			\EndIf
		\EndIf
	\State \Return True		
	\end{algorithmic}
\end{algorithm}

\begin{algorithm}[!htb]
    \caption{canOrientFixedEdges(Partial Orientation $(G,o')$, Set $D$, Arc $u\rightarrow v$)}\label{alg:orientFixedEdges}
    \begin{algorithmic}[1]
        \State // Recursively orient edges whose orientation is forced by Lemma~\ref{lemma:rules}
        \If{$u$ has two outgoing and no incoming arcs in $(G,o')$}
			\State Let $ux$ be the edge of $G$ which is unoriented in $(G,o')$
			\If{not canAddArcsRecursively($(G,o')$, $D$, $x\rightarrow u$)}
				\State \Return False
			\EndIf
		\EndIf
		\If{$v$ has two incoming and no outgoing arcs in $(G,o')$}
			\State Let $vx$ be the edge of $G$ which is unoriented in $(G,o')$
			\If{not canAddArcsRecursively($(G,o')$, $D$, $v\rightarrow x$)}
				\State \Return False
			\EndIf
		\EndIf
		\If{$uv \in D$}
			\If{not orientDeletable($(G,o')$, $D$, $u\rightarrow v$)}
				\State \Return False
			\EndIf
		\Else
			\If{not orientNonDeletable($(G,o')$, $D$, $u\rightarrow v$)
            \Statex[1]}
				\State \Return False
			\EndIf
		\EndIf
        \State \Return True
    \end{algorithmic}
\end{algorithm}

\begin{algorithm}[!htb]
	\caption{orientDeletable(Partial Orientation $(G,o')$, Set $D$, Arc $u\rightarrow v$)}
	\begin{algorithmic}[1]
        \State // Recursively orient edges whose orientation is forced by Lemma~\ref{lemma:rules} in the case that $uv\in D(G,o)$
		\ForAll{edges $ux$ incident to $u$ in $G$}
			\If{$ux\in D$}
				\If{not canAddArcsRecursively($(G,o')$, $D$, $x\rightarrow u$)}
					\State \Return False
				\EndIf
			\EndIf
		\EndFor
		\ForAll{edges $vx$ incident to $v$ in $G$}
			\If{$vx\in D$}
				\If{not canAddArcsRecursively($(G,o')$, $D$, $v\rightarrow x$)}
					\State \Return False
				\EndIf
			\EndIf
		\EndFor
		\State Let $ux_1, uy_1$ be the two edges incident with $u$ in $G$ such that $x_1, y_1\neq v$
		\If{$\{ux_1, uy_1\} \cap D = \emptyset$}
			\For{$z\in \{x_1,y_1\}$}
				\If{not canAddArcsRecursively($(G,o')$, $D$, $z\rightarrow u$)}
					\State \Return False
				\EndIf
			\EndFor
		\EndIf
		\State Let $vx_2, vy_2$ be the two edges incident with $v$ in $G$ such that $x_2, y_2\neq u$
		\If{$\{vx_2, vy_2\} \cap D = \emptyset$}
			\For{$z\in \{x_2,y_2\}$}
				\If{not canAddArcsRecursively($(G,o')$, $D$, $v\rightarrow z$)}
					\State \Return False
				\EndIf
			\EndFor
		\EndIf
		\State \Return True
	\end{algorithmic}
\end{algorithm}
\begin{algorithm}[!htb]
	\caption{orientNonDeletable(Partial Orientation $(G,o')$, Set $D$, Arc $u\rightarrow v$)}
	\begin{algorithmic}[1]
        \State // Recursively orient edges which are forced by Lemma~\ref{lemma:rules} in the case that $uv\not\in D(G,o)$
		\If{$u$ has precisely two incident arcs in $(G,o')$}
			\State Let $ux$ be the edge of $G$ which is unoriented in $(G,o')$
			\If{the arcs incident to $u$ in $(G,o')$ are one incoming and one outgoing}
				\If{not canAddArcsRecursively($(G,o')$, $D$, $u\rightarrow x$)}
					\State \Return False
				\EndIf
			\EndIf
		\EndIf
		\If{$v$ has precisely two incident arcs in $(G,o')$}
			\State Let $vx$ be the edge of $G$ which is unoriented in $(G,o')$
			\If{the arcs incident to $v$ in $(G,o')$ are one incoming and one outgoing}
				\If{not canAddArcsRecursively($(G,o')$, $D$, $x\rightarrow v$)}
					\State \Return False
				\EndIf
			\EndIf
		\EndIf
        \If{there exists an edge $ux$ of $G$ such that $x\neq v$ and $ux \not\in D$}
		      \If{not canAddArcsRecursively($(G,o')$, $D$, $u\rightarrow x$)}
			     \State \Return False
		      \EndIf
        \EndIf
        \If{there exists an edge $vy$ of $G$ such that $y\neq u$ and $vy\not\in D$}
		      \If{not canAddArcsRecursively($(G,o')$, $D$, $y\rightarrow v$)}
			     \State \Return False
		      \EndIf
        \EndIf
		\State \Return True
	\end{algorithmic}
\end{algorithm}
\clearpage

\subsection{Snarks for which Algorithm~\ref{alg:heuristic} is sufficient}\label{app:snarksPassingHeuristic}
\begin{table}[!htb]
    \centering
    \begin{tabular}{c|r|r}
        Order & Total & Passed \\\hline
        $10$ & $1$ & $0$\\
        $18$ & $2$ & $1$\\
        $20$ & $6$ & $6$\\
        $22$ & $31$ & $29$\\
        $24$ & $155$ & $152$\\
        $26$ & $1\,297$ & $1\,283$\\
        $28$ & $12\,517$ & $12\,472$\\
        $30$ & $139\,854$ & $139\,547$\\
        $32$ & $1\,764\,950$ & $1\,763\,302$\\
        $34$ & $25\,286\,953$ & $25\,273\,455$\\
        $36$ & $404\,899\,916$ & $404\,793\,575$
    \end{tabular}
    \caption{Number of cyclically $4$-edge-connected snarks for which Algorithm~\ref{alg:heuristic} is sufficient to decide that the graph has Frank number $2$. In the second column the total number of cyclically $4$-edge-connected snarks for the given order can be found. In the third column the number of such snarks in which the configuration of Theorem~\ref{thm:2OddCycles} or Theorem~\ref{thm:2Odd1Even} is present is given.}
    \label{tab:snarksPassingHeuristic}
\end{table}

\begin{table}[!htb]
    \centering
    \begin{tabular}{c|r | r}
        Order &  Algorithm~\ref{alg:heuristic} & Remainder\\\hline
         $28$ & $4$ s & $9$ s\\
         $30$ & $42$ s & $304$ s\\
         $32$ & $585$ s & $3$ h\\
         $34$ & $3$ h & $106$ h\\
         $36$ & $54$ h & $4975$ h
    \end{tabular}
    \caption{Runtimes of Algorithm~\ref{alg:heuristic} and~\ref{alg:fn=2} on the cyclically $4$-edge-connected snarks of order $28$ to $36$. In the second column, the runtime of Algorithm~\ref{alg:heuristic} on these graphs can be found for the specific order. In the third column, the runtime of Algorithm~\ref{alg:fn=2} \textbf{only} on the graphs for which Algorithm~\ref{alg:heuristic} failed can be found.}
    \label{tab:runtimes}
\end{table}

\end{document}